\documentclass{article}

\usepackage[a4paper]{geometry}

\usepackage[dvips]{hyperref}

\usepackage{pb-diagram}

\usepackage{curves}

\usepackage{amsmath}
\usepackage{amsthm}
\usepackage{amssymb}

\usepackage{bm}

\usepackage{latexsym}

   {\begin{proof}[Sketch of Proof of #1]}%
   {\end{proof}}

\theoremstyle{plain}
  \newtheorem{theorem}{Theorem}[section]
  \newtheorem{corollary}[theorem]{Corollary}
  
  \newtheorem{lemma}[theorem]{Lemma}
  
  \newtheorem{proposition}[theorem]{Proposition}
  \newtheorem{conjecture}[theorem]{Conjecture}

\theoremstyle{definition}
  \newtheorem{definition}[theorem]{Definition}

  \newtheorem{ex}[theorem]{Example}

  \newtheorem{remark}[theorem]{Remark}

  \newcommand{\Map}{\operatorname{Map}}

  \newcommand{\Ker}{\operatorname{Ker}}

  \newcommand{\Ima}{\operatorname{Im}}

  \newcommand{\shot}{\mathop\simeq\limits_S}
  \newcommand{\whot}{\mathop\simeq\limits_w}

  \newcommand{\quotient}[2]{%
  \left(#1\right)
  \hspace{-4pt}\raisebox{-5pt}{$\bigg/$}\hspace{-2pt}\raisebox{-12pt}{$#2$}%
  }
\newcommand{\rset}[2]{%
\left\{#1 \ \left| \ #2\right\}\right.
}
\newcommand{\lset}[2]{%
\left.\left\{#1 \ \right| \ #2\right\}
}

  \newcommand{\rarrow}[1]{\buildrel #1 \over \longrightarrow}

  \newcommand{\transpose}[1]{\raisebox{1ex}{$\scriptstyle t$}\kern-0.2ex #1}

  \newcommand{\Conv}{\operatorname{\mathrm{Conv}}} 

  \newcommand{\R}{\mathbb{R}}
  \newcommand{\C}{\mathbb{C}}

  \newcommand{\F}{\mathbb{F}}

  \newcommand{\Z}{\mathbb{Z}}

  \newcommand{\Sal}{\operatorname{Sal}} 

  \newcommand{\sgn}{\operatorname{\mathrm{sgn}}}
  \newcommand{\sign}{\operatorname{\mathrm{sign}}}

  \newcommand{\codim}{\operatorname{\mathrm{codim}}}

  \newcommand{\String}{\mathit{\String}}

\newcommand{\sk}{\operatorname{sk}}

\newcommand{\Sd}{\operatorname{\mathrm{Sd}}}

\newcommand{\fixhyperref}{%
\ifnum 42146=\euc"A4A2 \AtBeginDvi{}\else
\AtBeginDvi{}\fi}

\newcommand{\bibdir}{bib}

\title{\bfseries On the Homology of Configuration Spaces Associated to
Centers of Mass}
\author{Dai Tamaki
\thanks{Department of Mathematical Sciences, Shinshu University,
Matsumoto, 390-8621, Japan}
}
\date{}

\begin{document}
\maketitle

\begin{abstract}
 The aim of this paper is to make sample computations with the Salvetti
 complex of the ``center of mass'' arrangement introduced in
 \cite{math.AT/0611732} by Cohen and Kamiyama. We compute the homology
 of the Salvetti complex of these arrangements with coefficients in the
 sign representation of symmetric groups on $\F_p$ in the case of four
 particles. We show, when $p$ is an odd prime, the homology is
 isomorphic to the homology of the configuration space $F(\C,4)$ of
 distinct four points in $\C$  with the same coefficients. When $p=2$,
 we show the homology is different from that of $F(\C,4)$, hence obtain
 an alternative and more direct proof of a theorem of Cohen and Kamiyama
 in \cite{math.AT/0611732}.
\end{abstract}

\section{Introduction}
\label{introduction}

The configuration spaces of distinct points in $\C$
\[
F(\C,n)  = \lset{(z_1,\cdots,z_n)\in\C^n}{z_i\neq z_j \text{ if } i\neq j}
\]
have been playing important roles in homotopy theory. For example,
F.~Cohen \cite{F.Cohen83} used the weak equivalence
\[
 \Omega^2\Sigma^2 X \whot \quotient{\coprod_n
 F(\C,n)\times_{\Sigma_n}X^n}{\sim} 
\]
to construct an unstable splitting map
\begin{equation}
 \Sigma^{2n}\Omega^2\Sigma^2 X \longrightarrow
 \Sigma^{2n}F(\C,n)_+\wedge_{\Sigma_n} X^{\wedge n}.
 \label{Cohen_splitting}
\end{equation}
This is a desuspension of the well-known stable splitting 
\[
 \Sigma^{\infty}\Omega^2\Sigma^2 X \whot
 \Sigma^{\infty}\left(\bigvee_{n} F(\C,n)_+\wedge_{\Sigma_n}X^{\wedge
 n}\right) 
\]
due to Snaith \cite{Snaith74}.

These stable and unstable splitting maps can be used to construct
important maps in unstable homotopy theory. See \cite{Mahowald77,
F.Cohen83}, for example. It is also known that the map (\ref{Cohen_splitting})
cannot be desuspended further \cite{Cohen-Mahowald82}. There is a chance
of desuspending this   
map, however, if we localize at an appropriate prime. B.~Gray observed
in \cite{Gray93-1,Gray93-2} that if we could construct a map 
\begin{equation}
 \Sigma^2\Omega^2 S^{3} \longrightarrow \Sigma^2
 F(\C,p)_+\wedge_{\Sigma_p} S^p 
 \label{Gray_splitting}
\end{equation}
after localizing at an odd prime $p$, we would be able to refine results
of Cohen, Moore, and Neisendorfer
\cite{Cohen-Moore-Neisendorfer79-1,Cohen-Moore-Neisendorfer79-2} and
construct higher order EHP sequences. 

The difficulty is to construct a localized model of $\Omega^2S^3$ in
terms of configuration spaces. We do not know very much about
localizations of configuration spaces.
As an attempt to construct such a localized model, F.~Cohen and Kamiyama
introduced a subspace $M_{\ell}(\C,n)$ of $F(\C,n)$ in
\cite{math.AT/0611732}, for natural numbers $n$ and $\ell$ with
$\ell<n$. It can be defined as the 
complement in $\C^n$ of the complexification of the real central
hyperplane arrangement defined by
\[
\mathcal{C}_{n-1}^{\ell} = \{L_{I,J} \mid I,J\subset\{1,\cdots,n\},
|I|=|J|=\ell, I\neq J\},
\]
where 
\[
 L_{I,J} = \rset{(x_1,\cdots,x_n)\in\R^n}{ |J|\sum_{i\in I}x_i =
 |I|\sum_{j\in J}x_j }
\]
for $I,J \subset \{1,\cdots,n\}$.

Notice that
\[
 \mathcal{C}_{n-1}^{\ell} = 
 \begin{cases}
  \lset{L_{I,J}}{I,J\subset \{1,\cdots,n\},
 |I|=|J|\le \ell , I\cap J = \emptyset} & \ell\le\frac{n}{2}, \\
  \lset{L_{I,J}}{I,J\subset \{1,\cdots,n\},
 |I|=|J|\le n-\ell , I\cap J = \emptyset} & \ell>\frac{n}{2}
 \end{cases}
\]
and we have the duality
\[
  \mathcal{C}_{n-1}^{\ell} = \mathcal{C}_{n-1}^{n-\ell}.
\]
Thus we have the following inclusions of arrangements:
\[
 \mathcal{A}_{n-1} = \mathcal{C}_{n-1}^1 \subset \mathcal{C}_{n-1}^2
 \subset \cdots \subset \mathcal{C}_{n-1}^{[\frac{n}{2}]} \supset \cdots
 \supset \mathcal{C}_{n-1}^{n-2} \supset \mathcal{C}_{n-1}^{n-1} =
 \mathcal{A}_{n-1}. 
\]
By taking the complements, we obtain
\begin{equation}
 M_{\ell}(\C,n)\subset M_{\ell-1}(\C,n) \subset \cdots\subset M_2(\C,n)
 \subset M_1(\C,n) = F(\C,n) 
 \label{inclusion1}
\end{equation}
if $\ell\le [\frac{n}{2}]$. When $\ell> [\frac{n}{2}]$, we have
\begin{equation}
 M_{\ell}(\C,n) = M_{n-\ell}(\C,n) \subset M_{n-\ell-1}(\C,n) \subset
  \cdots\subset M_{2}(\C,n) 
 \subset M_{1}(\C,n) = F(\C,n). 
 \label{inclusion2}
\end{equation}
We denote these inclusions by
$i_{\ell,n} : M_{\ell}(\C,n) \hookrightarrow F(\C,n)$.

\begin{conjecture}[Cohen-Kamiyama]
 For an odd prime $p$, the natural inclusion
 \[
  i_{p,n} : M_p(\C,n) \hookrightarrow F(\C,n)
 \]
 induces an isomorphism
 \[
  (i_{p,n})_* : H_*(S_*(M_p(\C,n))\otimes_{\Sigma_n}\F_p(\pm 1))
 \longrightarrow H_*(S_*(F(\C,n))\otimes_{\Sigma_n}\F_p(\pm 1))
 \]
 for all $n$, where $S_*(-)$ is the singular chain complex functor,
 $\Sigma_n$ is the symmetric group on $n$ letters, and $\F_p(\pm 1)$
 is $\F_p$ regarded as a $\Sigma_n$-module via sign representation.
\end{conjecture}

They proved that the above conjecture implies the existence of the
desired map (\ref{Gray_splitting}). They also initiated the
analysis of the homology of $M_p(\C,n)$ and proved the statement of
the conjecture does not hold when $p=2$.

\begin{theorem}
 \label{fails_at_2}
 The class in $H_3(S_*(F(\C,4))\otimes_{\Sigma_4}\F_2)$ corresponding to
 $Q_1^2(x)$ in $H_*(\Omega^2S^{n+2};\F_2)$ is not in the image of the
 map 
 \[
  H_3(S_*(M_2(\C,4))\otimes_{\Sigma_4}\F_2) \longrightarrow
 H_3(S_*(F(\C,4))\otimes_{\Sigma_4}\F_2) 
 \]
 induced by the natural inclusion. Hence this map is not surjective.
\end{theorem}

Their method is indirect in the sense that they proved it by contraction
by calculating homology and cohomology operations. 
In order to find a way to attack the conjecture, a more direct method is
desirable.  

For a real central hyperplane arrangement $\mathcal{A}$ in general,
Salvetti \cite{Salvetti87} constructed a finite cell complex
$\Sal(\mathcal{A})$ embedded in the complement of the complexification
of $\mathcal{A}$ as a deformation retract. The aim of this paper is to
determine the maps induced on homology groups in the last step of the
inclusions (\ref{inclusion1}) and (\ref{inclusion2}) by using the
Salvetti complex. 

We reprove Theorem \ref{fails_at_2} by
analyzing the cellular structure of the Salvetti complex.
In fact, we show that
$H_3(S_*(M_2(\C,4))\otimes_{\Sigma_4}\F_2) \cong \F_2$ and the map takes
the generator to $0$.

For odd primes, we obtain the following result.

\begin{theorem}
 \label{p>=3}
 For an odd prime $p$, the inclusion $M_2(\C,4) \hookrightarrow F(\C,4)$
 induces an isomorphism
 \[
  H_*(S_*(M_2(\C,4))\otimes_{\Sigma_4}\F_p(\pm 1)) \cong
 H_*(S_*(F(\C,4))\otimes_{\Sigma_4}\F_p(\pm 1)). 
 \]
\end{theorem}

The paper is organized as follows:
\begin{itemize}
 \item Basic properties of the Salvetti complex used in this paper are
       recalled in \S\ref{Salvetti_complex}. 
 \item We describe the cell decomposition of the Salvetti complex for the
       braid arrangement in \S\ref{braid_arrangement} by using the
       notations in \cite{math.AT/0602085}.
 \item The Salvetti complex for the center of mass arrangement is
       studied in \S\ref{center}, including the computation of the
       homology of $\mathcal{C}_3^2$.  
\end{itemize}

\bigskip

\noindent\textbf{Acknowledgments}: The calculations in \S\ref{center}
were done while the author was visiting the
University of Aberdeen and the National University of Singapore. He
really appreciates the hospitality of the members of  
these institutes, especially, Jelena Grbic, Stephen Theriault,
Ran Levi, Jon Berrick, and Jie Wu. The author would also like to thank
the organizers of the MSI-SI conference ``Arrangements of Hyperplanes''
for invitation, during which the author learned the results in
\cite{Salvetti94} from Salvetti himself. It helped the author to 
understand and clarify the boundary formula proved in
\cite{math.AT/0602085}.

This work is partially supported by Grants-in-Aid for Scientific Research,
Ministry of Education, Culture, Sports, Science and Technology, Japan:
17540070

\section{Salvetti Complex and Oriented Matroid}
\label{Salvetti_complex}

Let us recall the definition and basic properties of the Salvetti
complex used this paper.

\subsection{Salvetti Complex for Real Central Arrangements}
\label{Salvetti_complex_basics}

A hyperplane arrangement $\mathcal{A}$ in a real vector space $V$
defines a stratification of $V$ by
\begin{eqnarray*}
 S_0 & = & V -\bigcup_{L\in\mathcal{A}} L \\
 S_1 & = & \bigcup_{L\in\mathcal{A}} L - \bigcup_{L,L'\in\mathcal{A}}
  L\cap L' \\
 & \vdots & \\
 S_{|\mathcal{L}|} & = & \bigcap_{L\in\mathcal{A}} L.
\end{eqnarray*}
Each stratum is a disjoint union of convex regions. These connected
components are called faces and the faces in the top stratum are called
chambers or topes. The set of faces is denoted by $\mathcal{L}(\mathcal{A})$
and the subset of chambers is denoted by
$\mathcal{L}^{(0)}(\mathcal{A})$. $\mathcal{L}(\mathcal{A})$ has a
structure of poset by
\[
 F\le G \Longleftrightarrow F\subset \overline{G}
\]
and called the face lattice.

When $\mathcal{A}$ is central, i.e.\ hyperplanes in $\mathcal{A}$ are
vector subspaces, 
Salvetti \cite{Salvetti87} constructed a
finite regular cell complex $\Sal(\mathcal{A})$ embedded in
$V\otimes\C -\bigcup_{L\in \mathcal{A}} L\otimes \C$ as a deformation
retract. Its cellular structure is determined by the face lattice
$\mathcal{L}(\mathcal{A})$ of the real arrangement $\mathcal{A}$.

A rough idea of the original construction of the Salvetti complex
$\Sal(\mathcal{A})$ is as follows:
 For each face $F$ in $\mathcal{A}$, choose a point $w(F)$ in $F$. For
 each pair of a face $F$ and a chamber $C$ with $F\le C$, define a point
 in $V\otimes\C$ by
 \[
  v(F,C) = w(F)\otimes 1 + (w(C)-w(F))\otimes i.
 \]
 The set of all such points is denoted by
 \[
 \sk_0\Sal(\mathcal{A}) = \{v(F,C) \mid F\in
 \mathcal{L}(\mathcal{A}), C\in \mathcal{L}^{(0)}(\mathcal{A}), F\le
 C\}.
 \]

The Salvetti complex $\Sal(\mathcal{A})$ is constructed as a Euclidean
simplicial complex embedded in
$V\otimes\C- \bigcup_{L\in\mathcal{A}} L\otimes \C$ by
forming simplices by choosing vertices from these points in a certain
manner. Salvetti defines a structure of a finite cell complex on
$\Sal(\mathcal{A})$ by combining several 
simplices together. Since all we need is this cell decomposition, we
recall this cellular structure, in stead of describing the rule for
simplices. The cell decomposition can be described in terms of the
face-chamber pairing, or the matroid product. The definition of the
matroid product can be found in \cite{Salvetti87,Arvola91}. An
alternative description will be given
later in Lemma \ref{MatroidProduct}. 

\begin{definition}
 For $F \in \mathcal{L}(\mathcal{A})$ and $C \in
 \mathcal{L}^{(0)}(\mathcal{A})$ with $F\le C$, define a subset of
 $\sk_0(\Sal(\mathcal{A}))$ by
 \[
  \mathcal{D}(F,C) = \{v(G,G\circ C) \mid G\ge F\},
 \]
 where $\circ$ is the matroid product.
 This set is regarded as a poset by
 \[
  v(G,G\circ C) \le v(H,H\circ C) \Longleftrightarrow G\le H.
 \]
 The (geometric realization of the) order complex of $\mathcal{D}(F,C)$
 is denoted by $D(F,C)$. 

\end{definition}

\begin{lemma}
 \label{D(F,C)}
 The complex $D(F,C)$ has the following properties:
 \begin{enumerate}
  \item The inclusion of vertices induces a simplicial embedding
	\[
	 D(F,C) \hookrightarrow \Sal(\mathcal{A}).
	\]
  \item $D(F,C)$ is homeomorphic to a disk of dimension $\codim F$.
  \item The boundary of $D(F,C)$ is given by
	\[
	 \partial D(F,C) = \bigcup_{G>F} D(G,G\circ C).
	\]
  \item The decomposition
	\[
	 \Sal(\mathcal{A}) = \bigcup_{v(F,C)\in
	\sk_0(\Sal(\mathcal{A}\otimes\C))} (D(F,C)-\partial D(F,C))
	\]
	defines a structure of a finite regular cell complex on
	$\Sal(\mathcal{A})$. 
 \end{enumerate}
\end{lemma}

In order to compute the boundary, therefore, we need to understand the
matroid product. The following elementary fact is very useful.

\begin{lemma}
 \label{FaceLattice}

 Given a real central hyperplane arrangement
 $\mathcal{A} = \{L_1,\cdots, L_n\}$ in a real inner product space
 $(V,\langle-,-\rangle)$, 
 choose a normal vector $a_i$ for each hyperplane $L_i$ 
 \[
  L_i = \{x\in V \mid \langle a_i,x\rangle = 0\}
 \]
 and define $\mathcal{V}(\mathcal{A}) = \{a_1,\cdots, a_n\}$.

 Let $S_1=\{0,+1,-1\}$ be the poset with $0<+1,-1$. For $F \in
 \mathcal{L}(\mathcal{A})$, define 
 \[
  \tau_F : \mathcal{V}(\mathcal{A}) \longrightarrow S_1
 \]
 by
 \[
  \tau_F(a) = \sign \langle a,F\rangle,
 \]
 where
 \[
  \sign : \R \longrightarrow S_1
 \]
 is the sign function.

 Then we obtain an embedding
 \[
  \tau : \mathcal{L}(\mathcal{A}) \hookrightarrow
 \Map(\mathcal{V}(\mathcal{A}), S_1).
 \]
\end{lemma}

It is possible to give an explicit description of the subset
$\tau(\mathcal{L}(\mathcal{A}))$ by using the language of oriented matroid.
See the paper \cite{Gelfand-Rybnikov89} by Gel$'$fand and Rybnikov or
the book on oriented matroids \cite{OrientedMatroids} by five authors for
more details. 

With this identification, the face-chamber pairing in the face lattice
$\mathcal{L}(\mathcal{A})$ can be translated into the following product
of functions.

\begin{lemma}
 \label{MatroidProduct}
 Let $E$ be a set. For $\varphi, \psi \in \Map(E,S_1)$ define
 $\varphi\circ\psi \in \Map(E,S_1)$ by
 \[
  (\varphi\circ\psi)(x) = \begin{cases}
			   \varphi(x), & \varphi(x)\neq 0 \\
			   \psi(x), & \varphi(x) = 0.
			  \end{cases}
 \]

 Then for a real central arrangement $\mathcal{A}$, we have
 \[
  \tau_{F}\circ\tau_{G} = \tau_{F\circ G}
 \]
 for $F, G \in \mathcal{L}(\mathcal{A})$.
\end{lemma}

We may formally complexify the poset $\mathcal{L}(\mathcal{A})$ by using
the poset $S_2=\{0,+1,-1,+i,-i\}$ with ordering $0<\pm 1<\pm i$.

\begin{definition}
 Define $\Sigma_2$-equivariant inclusions
 \[
  i_1,i_2 : S_1 \hookrightarrow S_2
 \]
 by
 \begin{eqnarray*}
  i_1(0) & = & i_2(0) = 0 \\
  i_1(\pm 1) & = & \pm 1 \\
  i_2(\pm 1) & = & \pm i.
 \end{eqnarray*}
\end{definition}

\begin{definition}
 Let $E$ be a set and $L$ be a subset of $\Map(E,S_1)$. Define a
 subposet $L\otimes\C$ of $\Map(E,S_2)$ by
 \[
  L\otimes\C = \lset{(i_1)_*(X)\circ (i_2)_*(Y)}{X,Y\in L},
 \]
 where
 \[
  (i_1)_*, (i_2)_* : \Map(E,S_1) \longrightarrow \Map(E,S_2)
 \]
 are the maps induced by $i_1$ and $i_2$ and the matorid product $\circ$
 in $\Map(E,S_2)$ is defined by
 \[
  (\varphi\circ \psi)(x) = \begin{cases}
			    \varphi(x), & \varphi(x)\not\le \psi(x) \\
			    \psi(x), & \varphi(x)\le \psi(x).
			   \end{cases}
 \]
\end{definition}

Another useful way of describing the Salvetti complex is as follows. See
\cite{Bjorner-Ziegler92}, for example.

\begin{proposition}
 \label{Salvetti_by_poset}
 Let $\mathcal{A}$ be a real central arrangement in an inner product
 space $V$ and $\mathcal{V}(\mathcal{A})$ be a set of unit normal
 vectors of hyperplanes in $\mathcal{A}$. Define a subposet
 $\mathcal{L}^{(1)}(\mathcal{A})$ of $\mathcal{L}(\mathcal{A})\otimes\C$
 by
 \[
  \mathcal{L}^{(1)}(\mathcal{A}) = \lset{X\in
 \mathcal{L}(\mathcal{A})\otimes\C}{X(v)\neq 0 \text{ for all } v\in
 \mathcal{V}(\mathcal{A})}.  
 \]
 Then $\mathcal{L}^{(1)}(\mathcal{A})$ is isomorphic to the face poset
 $F(\Sal(A))$ of the Salvetti complex of $\mathcal{A}$ as posets. Thus
 we have an isomorphism of simplicial complexes
 \[
  B\mathcal{L}^{(1)}(\mathcal{A}) \cong \Sd(\Sal(A)),
 \]
 where $B(-)$ is the classifying space (order complex) functor and
 $\Sd(-)$ is the barycentric subdivision. 
\end{proposition}
\subsection{Salvetti Complex for Reflection Arrangements}
\label{Salvetti_complex_reflection}

When $\mathcal{A}$ is a reflection arrangement, Salvetti analyzed the
cellular structure of $\Sal(\mathcal{A})/G(\mathcal{A})$ in
\cite{Salvetti94}, where $G(\mathcal{A})$ is the reflection group
associated with $\mathcal{A}$. In particular, he described
the boundary in the cellular cochain complex.
Salvetti's work includes the case of affine reflection groups. 
Here we only consider central arrangements.

Let $\mathcal{A}$ be a real central reflection arrangement in $V$ and
$\mathcal{L}(\mathcal{A})$ be the face poset. We have a cellular
decomposition (stratification) of $V$
\[
 V = \coprod_{F\in\mathcal{L}(\mathcal{A})} F.
\]
The cell complex dual to this cellular decomposition is denoted by
$C(\mathcal{A})$, whose face poset is denoted by
$F(C(\mathcal{A}))$. One of the ways to construct $C(\mathcal{A})$ is to
choose a chamber $C_0\in \mathcal{L}^{(0)}(\mathcal{A})$ and a point
$v_0$ inside of $C_0$ and to take the convex hull of the
$G$-orbit of $v_0$
\[
 C(\mathcal{A}) = \Conv(gv_0 \mid g\in G(\mathcal{A})).
\]
The cellular structure of $C(\mathcal{A})$ is given by that of this
convex polytope.

When $e\in F(C(\mathcal{A}))$ is the face dual to
$F\in \mathcal{L}(\mathcal{A})$ we denote $e=F^*$ and $F=e^*$.

\begin{definition}
 For each face $e\in F(C(\mathcal{A}))$, let
 $\gamma(e)\in G(\mathcal{A})$ be the unique element of minimal length
 such that
 \[
  \gamma(e)^{-1}(e^*) \subset \overline{C_0}.
 \]
\end{definition}

Salvetti found the following description of
$\Sal(\mathcal{A})/G(\mathcal{A})$. 

\begin{theorem}
 \label{Sal/G}
 The cell complex $\Sal(\mathcal{A})/G(\mathcal{A})$ can be identified
 with the cell complex given by identifying those two cells $e,e'$ in
 $C(\mathcal{A})$ which are in the same $G(\mathcal{A})$-orbit by using
 the homeomorphism induced by the element $\gamma(e')\gamma(e)^{-1}$.
\end{theorem}

In order to describe the boundary homomorphism in the cellular chain
complex, Salvetti defined an orientation on each cell in
$C(\mathcal{A})$. 

\begin{definition}
 Fix an ordering of hyperplanes $H_1, \cdots, H_n$ bounding the chosen
 chamber $C_0$. Let $v_i$ be the projection of $v_0$ onto $H_i$.
 Define
 \begin{eqnarray*}
  F(C_0) & = & \lset{F\in \mathcal{L}(\mathcal{A})}{F \subset
   \overline{C}_0}, \\
  F^*(C_0) & = & \lset{F^*\in F(C(\mathcal{A}))}{F\in F(C_0)}.  
 \end{eqnarray*}

 For a cell $e \in F^*(C_0)$, define an orientation on $e$ as
 follows. Let $H_{i_1}, \cdots, H_{i_k}$ be hyperplanes with 
 \[
  e^* \subset H_{i_1} \cap \cdots\cap H_{i_k}
 \]
 and $i_1<\cdots <i_k$. The orientation of $e$ is induced from the
 ordering $v_0,v_{i_1},\cdots, v_{i_k}$ under the inclusion
 \[
  \Conv(v_0,v_{i_1},\cdots,v_{i_k}) \subset \overline{e}.
 \]
 In general, define an orientation on $e \in C(\mathcal{A})$ in such a
 way that $\gamma(e)^{-1}$ is orientation preserving. 
\end{definition}

Under the above orientations, the incidence numbers among cells in
$C(\mathcal{A})$ are described as follows.

\begin{proposition}
 \label{incidence_number}
 Let $F\in F(C_0)$ and $G\in \mathcal{L}(\mathcal{A})$ with
 \[
  \overline{g G} \supset F
 \]
 for an element $g\in G(\mathcal{A})$ of the shortest length and
 $\dim G=\dim F+1$. Then  
 \[
  [F^*:G^*] = (-1)^{\ell(g)}[F^*,(gG)^*].
 \]
\end{proposition}

\subsection{Maps between Salvetti complexes}
\label{maps_between_Salvetti}

The ``center of mass'' arrangement $\mathcal{C}_{n-1}^{\ell}$ is
obtained by adding hyperplanes to the braid arrangement
$\mathcal{A}_{n-1}$. And we have a contravariant inclusion
\[
 M_{\ell}(\C,n) \hookrightarrow F(\C,n)
\]
of the complements.
We also have corresponding maps on the Salvetti complexes.

\begin{lemma}
 Let $\mathcal{A}$ be a real central arrangement in a vector space $V$
 and $\mathcal{B}\subset\mathcal{A}$ be a subarrangement. Then the
 inclusion 
 \[
  i : \mathcal{B}\hookrightarrow \mathcal{A}
 \]
 induces a cellular map
 \[
  i^* : \Sal(\mathcal{A}) \longrightarrow \Sal(\mathcal{B})
 \]
 which makes the following diagram commutative up to homotopy
 \[
  \begin{diagram}
   \node{\Sal(\mathcal{A})} \arrow{e,t}{i^*} \arrow{s,J}
   \node{\Sal(\mathcal{B})} \arrow{s,J} \\
   \node{V\otimes\C-\bigcup_{H\in\mathcal{A}}H\otimes\C} \arrow{e}
   \node{V\otimes\C-\bigcup_{H\in\mathcal{B}}H\otimes\C}
  \end{diagram}
 \]
\end{lemma}

\begin{proof}
 The inclusion $i : \mathcal{B} \hookrightarrow \mathcal{A}$ induces a
 map of face lattices
 \[
  i^* : \mathcal{L}(\mathcal{A}) \longrightarrow
 \mathcal{L}(\mathcal{B}) 
 \]
 (i.e.\ it induces a strong map between oriented matroids) which induces
 a map of posets
 \[
  i^* : \mathcal{L}(\mathcal{A})\otimes\C \longrightarrow
 \mathcal{L}(\mathcal{B})\otimes\C.
 \]
 Since $i^*$ is given by restriction, we obtain
 \[
  i^* : \mathcal{L}^{(1)}(\mathcal{A}) \longrightarrow
 \mathcal{L}^{(1)}(\mathcal{B}) 
 \]
 and hence a map
 \[
  i^* : \Sal(\mathcal{A}) \longrightarrow \Sal(\mathcal{B})
 \]
 by Proposition \ref{Salvetti_by_poset}.

 The embeddings of the Salvetti complexes depend on choices of
 simplicial vertices corresponding to faces in the face lattices. We
 obtain embeddings by choosing $w(F)$ for
 $F\in \mathcal{L}(\mathcal{A})$ first and then by choosing vertices for
 $\mathcal{B}$ among $\{w(F)\}$ which make the required diagram
 commutative. 
\end{proof}

We would like to know the behavior of the chain map
\[
 i_* : C_*(\Sal(\mathcal{A})) \longrightarrow C_*(\Sal(\mathcal{B})).
\]

\begin{lemma}
 \label{i_is_surjective}

 Let $\mathcal{A}$ be a real central arrangement in $V$ and
 $\mathcal{B}\subset\mathcal{A}$ be a subarrangement. Then the inclusion
 \[
  i : \mathcal{B}\hookrightarrow \mathcal{A}
 \]
 induces a surjective chain map
 \[
 i^* : C_*(\Sal(\mathcal{A})) \longrightarrow
 C_*(\Sal(\mathcal{B})).
 \]
\end{lemma}

\begin{proof}
 Generators of $C_*(\Sal(\mathcal{B}))$ are in one-to-one correspondence
 with pairs $(F,C)$ of a face $F$ and a chamber $C$ in
 $\mathcal{L}(\mathcal{B})$. Since
 \[
  i^* : \mathcal{L}(\mathcal{A}) \longrightarrow
 \mathcal{L}(\mathcal{B}) 
 \]
 is surjective, it induces a surjective map on the cellular chain
 complexes of Salvetti complexes.
\end{proof}

\section{The Salvetti Complex for the Braid Arrangement}
\label{braid_arrangement}

We need to understand the cellular structure of the Salvetti complex of
the braid arrangement in order to compare it with that of the center of
mass arrangement.

\subsection{The Structure of Cell Complex}
\label{braid_arrangement_cells}

The braid arrangement is a typical example of reflection arrangements
and the results of \S\ref{Salvetti_complex_reflection} apply.
In particular, the cell structure of the Salvetti complex for the braid
arrangement can be described in terms of partitions. The following
symbols are introduced in \cite{math.AT/0602085}.

\begin{definition}
 A partition of $\{1,\cdots, n\}$ is a surjective map
 \[
 \lambda : \{1,\cdots,n\} \longrightarrow \{1,\cdots,n-r\}
 \]
 for some $0\le r < n$. The number $r$ is called the rank of this
 partition.
 
 The set of partitions of $\{1,\cdots,n\}$ is denoted by
 $\Pi_n$. The subset of rank $r$ partitions is denoted by
 $\Pi_{n,r}$. $\Pi_n$ becomes a poset under refinement.
 Note that rank $0$ partitions are nothing but elements of $\Sigma_n$.
\end{definition}

\begin{definition}
 \label{CubicalSymbol}
 For a partition $\lambda \in \Pi_n$ of rank $r$ and
 $\sigma \in \Sigma_n$ with $\sigma\ge \lambda$, define a symbol
 $S(\lambda,\sigma)$ as follows: 
 \begin{enumerate}
  \item For each $1\le i\le n-r$, draw vertically stacked squares $S_i$
	of length $|\lambda^{-1}(i)|$.
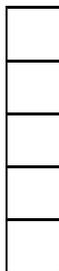
	\begin{center}
	 \begin{picture}(20,100)(0,0)
	  \put(0,0){\framebox(20,20)}
	  \put(0,20){\framebox(20,20)}
	  \put(0,40){\framebox(20,20)}
	  \put(0,60){\framebox(20,20)}
	  \put(0,80){\framebox(20,20)}
	 \end{picture}
	\end{center}
  \item Order $\lambda^{-1}(i)$ according to $\sigma$ and label each
	square in $S_i$  from bottom to top by elements in
	$\lambda^{-1}(i)$. For example, when
	$\lambda^{-1}(i) = \{i_1, i_2, i_3, i_4, i_5\}$ and if these
	numbers appear in $(\sigma(1),\cdots,\sigma(n))$ in the order
	\[
	 i_1,i_2,i_3,i_4,i_5
	\]
	then $S_i$ is labeled as
	
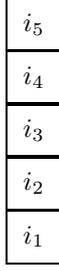
\begin{center}
	 \begin{picture}(20,100)(0,0)
	  \put(0, 0){\framebox(20,20){$i_1$}}
	  \put(0,20){\framebox(20,20){$i_2$}}
	  \put(0,40){\framebox(20,20){$i_3$}}
	  \put(0,60){\framebox(20,20){$i_4$}}
	  \put(0,80){\framebox(20,20){$i_5$}}
	 \end{picture}
	\end{center}
  \item Place $S_1, \cdots, S_{n-r}$ side by side from left to
	right. $S(\lambda,\sigma)$ is the resulting picture.
	
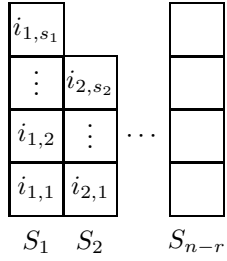
\begin{center}
	 \begin{picture}(80,100)(0,-20)
	  \put(0,-20){\makebox(20,20){$S_1$}}
	  \put(0, 0){\framebox(20,20){$i_{1,1}$}}
	  \put(0,20){\framebox(20,20){$i_{1,2}$}}
	  \put(0,40){\framebox(20,20){}}
	  \put(10,53){\makebox(0,0){$\vdots$}}
	  \put(0,60){\framebox(20,20){$i_{1,s_1}$}}

	  \put(20,-20){\makebox(20,20){$S_2$}}
	  \put(20, 0){\framebox(20,20){$i_{2,1}$}}
	  \put(20,20){\framebox(20,20){}}
	  \put(30,33){\makebox(0,0){$\vdots$}}
	  \put(20,40){\framebox(20,20){$i_{2,s_2}$}}

	  \put(40,20){\makebox(20,20){$\cdots$}}

	  \put(60,-20){\makebox(20,20){$S_{n-r}$}}
	  \put(60, 0){\framebox(20,20){}}
	  \put(60,20){\framebox(20,20){}}
	  \put(60,40){\framebox(20,20){}}
	  \put(60,60){\framebox(20,20){}}
	 \end{picture}
	\end{center}
 \end{enumerate}
\end{definition}

The following observation played an essential role in
\cite{math.AT/0602085}. 

\begin{lemma}
 \label{VertexSymbols}
 There is a bijection between the set of vertices (of the simplicial
 structure) $\sk_0\Sal(\mathcal{A}_{n-1})$ and the set of symbols
 \[
  \{S(\lambda,\sigma) \mid \lambda \in \Pi_n, \sigma\in\Sigma_n, 
 \lambda\le \sigma\}.
 \]

 Thus we obtain a bijection between the cells of
 $\Sal(\mathcal{A}_{n-1})$ and the symbols $S(\lambda, \sigma)$.
\end{lemma}

In order to compute the boundary maps of the cellular chain complex of
$\Sal(\mathcal{A}_{n-1})$, we need to fix orientations of cells. We
follow the orientations defined in \S\ref{Salvetti_complex_reflection}.

We choose the chamber
\[
 C_0 = \lset{(x_1,\cdots,x_n)\in\R^n}{x_1<x_2<\cdots<x_n}
\]
and define $v_0=(1,2,\cdots,n)\in C_0$. Then 
\[
 C(\mathcal{A}_{n-1}) = \Conv(v_0\sigma \mid \sigma\in\Sigma_n).
\]

We have the following refinement of Theorem \ref{Sal/G}.

\begin{proposition}
 We have the following isomorphism of $\F_p$-modules
 \begin{eqnarray*}
  C_*(\Sal(\mathcal{A}_{n-1}))\otimes_{\Sigma_n}\F_p(\pm 1) & \cong &
 \F_p\left\langle [e] \mid [e] \in F(C(\mathcal{A}_{n-1}))/\Sigma_n)
 \right\rangle \\
  & \cong & 
   \F_p\left\langle e \mid e \in F^*(C_0) \right\rangle.
 \end{eqnarray*}
\end{proposition}

\begin{proof}
 Let $S$ be a set with an action of $\Sigma_n$. Then we have an
 isomorphism of $\F_p$-modules
 \[
  \Z\langle S\rangle\otimes_{\Sigma_n}\F_p(\pm 1) \cong \Z\langle
 S/\Sigma_n\rangle \otimes \F_p.
 \]
 And the result follows by the idenfitications of cells in the proof of
 Theorem \ref{Sal/G} in \cite{Salvetti94}.
\end{proof}

The bounding hyperplanes of $C_0$ are
\[
 L_{1,2}, L_{2,3}, \cdots, L_{n,n-1}.
\]
This ordering of hyperplanes determines orientations of cells in
$C(\mathcal{A}_{n-1})$. Under the correspondence in Lemma
\ref{VertexSymbols}, cells in $C(\mathcal{A}_{n-1})$ correspond to
$S(\lambda,1)$ with $\lambda\in\Pi_n$ and $\lambda\le 1$. Those cells in
$F^*(C_0)$ corresponds to ordered partitions.

\begin{definition}
 An order preserving surjective map
 \[
  \lambda : \{1,\cdots,k\} \longrightarrow \{1,\cdots,k-r\}
 \]
 is called an ordered partition of rank $r$. The set of ordered
 partitions of $\{1,\cdots,k\}$ of rank $r$ is denoted by $O_{k,r}$. 
\end{definition}

\begin{corollary}
 \label{cells_by_ordered_partitions}
 Under the identification in Lemma \ref{VertexSymbols}, we have the
 following isomorphism of $\F_p$-modules 
 \[
  C_s(\Sal(\mathcal{A}_{n-1}))\otimes_{\Sigma_n}\F_p(\pm 1) \cong
 \F_p\langle D(\lambda,(1|\cdots|n)) \mid \lambda\in O_{n,n-s} \rangle.
 \]
\end{corollary}

The following formula for the boundary map follows from Proposition
\ref{incidence_number}.

\begin{lemma}
 \label{boundary_formula}
 For $\lambda\in O_{n,n-s}$, we have
 \[
  \partial(D(\lambda,(1|\cdots|n))) = \sum_{\tau\in O_{n,n-s-1},
 \lambda<\tau} \sum_{\sigma\in\Sigma(\lambda)} \sgn(\sigma)
 [D(\lambda,(1|\cdots|n)):D(\tau,(1|\cdots|n))]D(\tau\sigma,\sigma)
 \]
 in $C_*(\Sal(\mathcal{A}_{n-1}))$, where
 \[
  \Sigma(\lambda) = \lset{\sigma\in\Sigma_n}{\lambda\sigma = \lambda}
 \]
 is the set of permutations preserving the partition $\lambda$.
\end{lemma}

\begin{proof}
 The set of faces in $\Sal(\mathcal{A}_{n-1})$ contained in
 $D(\lambda,(1|\cdots|n))$ as a face is given by $D(\tau\sigma,\sigma)$
 for $\lambda<\tau$ and $\sigma\in\Sigma(\lambda)$. Thus
 \begin{eqnarray*}
  \partial D(\lambda,(1|\cdots|n)) & = & \sum_{\tau\in
   O_{n,n-s-1}\lambda<\tau} \sum_{\sigma\in\Sigma(\tau)}
   [D(\lambda,(1|\cdots|n)):D(\tau\sigma,\sigma)]
   D(\tau\sigma,\sigma) \\
  & = & \sum_{\tau\in
   O_{n,n-s-1}\lambda<\tau} \sum_{\sigma\in\Sigma(\tau)}
  \sgn(\sigma) [D(\lambda,(1|\cdots|n)):D(\tau,(1|\cdots|n))]
   D(\tau\sigma,\sigma)
 \end{eqnarray*}
 by Proposition \ref{incidence_number}.
\end{proof}

Note that the incidence number
$[D(\lambda,(1|\cdots|n)):D(\tau,(1|\cdots|n))]$ can be determined by
comparing the ``positions of $=$'' in $\lambda$ and $\tau$.
\subsection{The Homology of $F(\C,4)$}
\label{FC4}

The homology $H_*(S_*(F(\C,4))\otimes_{\Sigma_4}\F_p(\pm 1))$ is
well-known. We need, 
however, an explicit description in order to compare it with
$H_*(C_*(\Sal(\mathcal{C}^2_3))\otimes_{\Sigma_4}\F_p(\pm 1))$ in the
next section.

Let us compute
$H_*(C_*(\Sal(\mathcal{A}_3))\otimes_{\Sigma_4}\F_p(\pm 1))$ by using the 
symbols introduced in the previous section. $\Sal(\mathcal{A}_3)$ has
the following cells: 
\begin{itemize}
 \item $0$-cells are in one-to-one correspondence with the symbols
       \begin{center}
	\begin{picture}(80,20)(0,0)
	 \put(0, 0){\framebox(20,20){$\sigma(1)$}}
	 \put(20, 0){\framebox(20,20){$\sigma(2)$}}
	 \put(40, 0){\framebox(20,20){$\sigma(3)$}}
	 \put(60, 0){\framebox(20,20){$\sigma(4)$}}
	\end{picture}
       \end{center}
 \item $1$-cells are in one-to-one correspondence with the symbols
       \begin{center}
	\begin{picture}(60,40)(0,0)
	 \put(0, 0){\framebox(20,20){$\sigma(1)$}}
	 \put(0, 20){\framebox(20,20){$\sigma(2)$}}
	 \put(20, 0){\framebox(20,20){$\sigma(3)$}}
	 \put(40, 0){\framebox(20,20){$\sigma(4)$}}
	\end{picture},
	\begin{picture}(60,40)(0,0)
	 \put(0, 0){\framebox(20,20){$\sigma(1)$}}
	 \put(20, 0){\framebox(20,20){$\sigma(2)$}}
	 \put(20, 20){\framebox(20,20){$\sigma(3)$}}
	 \put(40, 0){\framebox(20,20){$\sigma(4)$}}
	\end{picture},
	\begin{picture}(60,40)(0,0)
	 \put(0, 0){\framebox(20,20){$\sigma(1)$}}
	 \put(20, 0){\framebox(20,20){$\sigma(2)$}}
	 \put(40, 0){\framebox(20,20){$\sigma(3)$}}
	 \put(40, 20){\framebox(20,20){$\sigma(4)$}}
	\end{picture}
       \end{center}
 \item $2$-cells are in one-to-one correspondence with the symbols
       \begin{center}
	\begin{picture}(40,60)(0,0)
	 \put(0, 0){\framebox(20,20){$\sigma(1)$}}
	 \put(20, 0){\framebox(20,20){$\sigma(2)$}}
	 \put(20, 20){\framebox(20,20){$\sigma(3)$}}
	 \put(20, 40){\framebox(20,20){$\sigma(4)$}}
	\end{picture},
	\begin{picture}(40,40)(0,0)
	 \put(0, 0){\framebox(20,20){$\sigma(1)$}}
	 \put(0, 20){\framebox(20,20){$\sigma(2)$}}
	 \put(20, 0){\framebox(20,20){$\sigma(3)$}}
	 \put(20, 20){\framebox(20,20){$\sigma(4)$}}
	\end{picture},
	\begin{picture}(40,60)(0,0)
	 \put(0, 0){\framebox(20,20){$\sigma(1)$}}
	 \put(0, 20){\framebox(20,20){$\sigma(2)$}}
	 \put(0, 40){\framebox(20,20){$\sigma(3)$}}
	 \put(20, 0){\framebox(20,20){$\sigma(4)$}}
	\end{picture}.
       \end{center}
 \item $3$-cells are in one-to-one correspondence with the symbols
       \begin{center}
	\begin{picture}(20,80)(0,0)
	 \put(0, 0){\framebox(20,20){$\sigma(1)$}}
	 \put(0, 20){\framebox(20,20){$\sigma(2)$}}
	 \put(0, 40){\framebox(20,20){$\sigma(3)$}}
	 \put(0, 60){\framebox(20,20){$\sigma(4)$}}
	\end{picture}
       \end{center}
\end{itemize}

Thus the chain complex 
$C_*(\Sal(\mathcal{A}_3))\otimes_{\Sigma_4}\F_p(\pm 1)$ has the
following basis. 
\begin{eqnarray*}
 C_0(\Sal(\mathcal{A}_3))\otimes_{\Sigma_4}\F_p(\pm 1)
  & = & \left\langle
	 \raisebox{-6pt}{\begin{picture}(64,15)(-2,0)
	 \put(0, 0){\framebox(15,15){$1$}}
	 \put(15, 0){\framebox(15,15){$2$}}
	 \put(30, 0){\framebox(15,15){$3$}}
	 \put(45, 0){\framebox(15,15){$4$}}
	 \end{picture}} 
		\right\rangle \\
 C_1(\Sal(\mathcal{A}_3))\otimes_{\Sigma_4}\F_p(\pm 1)
  & = & \left\langle 
	\raisebox{-12pt}{\begin{picture}(49,30)(-2,0)
	 \put(0, 0){\framebox(15,15){$1$}}
	 \put(0, 15){\framebox(15,15){$2$}}
	 \put(15, 0){\framebox(15,15){$3$}}
	 \put(30, 0){\framebox(15,15){$4$}}
	\end{picture}, 
	\begin{picture}(49,30)(-2,0)
	 \put(0, 0){\framebox(15,15){$1$}}
	 \put(15, 0){\framebox(15,15){$2$}}
	 \put(15, 15){\framebox(15,15){$3$}}
	 \put(30, 0){\framebox(15,15){$4$}}
	\end{picture}, 
	\begin{picture}(49,30)(-2,0)
	 \put(0, 0){\framebox(15,15){$1$}}
	 \put(15, 0){\framebox(15,15){$2$}}
	 \put(30, 0){\framebox(15,15){$3$}}
	 \put(30, 15){\framebox(15,15){$4$}}
	\end{picture}} 
	\right\rangle \\
 C_2(\Sal(\mathcal{A}_3))\otimes_{\Sigma_4}\F_p(\pm 1)
  & = & \left\langle 
	\raisebox{-18pt}{\begin{picture}(34,45)(-2,0)
	 \put(0, 0){\framebox(15,15){$1$}}
	 \put(15, 0){\framebox(15,15){$2$}}
	 \put(15, 15){\framebox(15,15){$3$}}
	 \put(15, 30){\framebox(15,15){$4$}}
	\end{picture}, 
	\begin{picture}(34,30)(-2,0)
	 \put(0, 0){\framebox(15,15){$1$}}
	 \put(0, 15){\framebox(15,15){$2$}}
	 \put(15, 0){\framebox(15,15){$3$}}
	 \put(15, 15){\framebox(15,15){$4$}}
	\end{picture}, 
	\begin{picture}(34,45)(-2,0)
	 \put(0, 0){\framebox(15,15){$1$}}
	 \put(0, 15){\framebox(15,15){$2$}}
	 \put(0, 30){\framebox(15,15){$3$}}
	 \put(15, 0){\framebox(15,15){$4$}}
	\end{picture}} 
	\right\rangle \\
 C_3(\Sal(\mathcal{A}_3))\otimes_{\Sigma_4}\F_p(\pm 1)
  & = & \left\langle
	 \raisebox{-24pt}{\begin{picture}(19,60)(-2,0)
	 \put(0, 0){\framebox(15,15){$1$}}
	 \put(0, 15){\framebox(15,15){$2$}}
	 \put(0, 30){\framebox(15,15){$3$}}
	 \put(0, 45){\framebox(15,15){$4$}}
	 \end{picture}}
	\right\rangle.
\end{eqnarray*}

The boundaries can be computed by using Lemma \ref{boundary_formula} as
follows. For $0$-cells, we obviously have
\[
  \partial_0\left(\raisebox{-6pt}{\begin{picture}(64,15)(-2,0)
	 \put(0, 0){\framebox(15,15){$1$}}
	 \put(15, 0){\framebox(15,15){$2$}}
	 \put(30, 0){\framebox(15,15){$3$}}
	 \put(45, 0){\framebox(15,15){$4$}}
	 \end{picture}}\right) = 0.
\]
For $1$-cells, we have
\begin{eqnarray*}
 \partial_1\left(\raisebox{-12pt}{\begin{picture}(49,30)(-2,0)
	 \put(0, 0){\framebox(15,15){$1$}}
	 \put(0, 15){\framebox(15,15){$2$}}
	 \put(15, 0){\framebox(15,15){$3$}}
	 \put(30, 0){\framebox(15,15){$4$}}
	\end{picture}}\right) & = & 
 \raisebox{-4pt}{\begin{picture}(64,15)(-2,0)
	 \put(0, 0){\framebox(15,15){$1$}}
	 \put(15, 0){\framebox(15,15){$2$}}
	 \put(30, 0){\framebox(15,15){$3$}}
	 \put(45, 0){\framebox(15,15){$4$}}
 \end{picture}} -  
 \raisebox{-4pt}{\begin{picture}(64,15)(-2,0)
	 \put(0, 0){\framebox(15,15){$2$}}
	 \put(15, 0){\framebox(15,15){$1$}}
	 \put(30, 0){\framebox(15,15){$3$}}
	 \put(45, 0){\framebox(15,15){$4$}}
 \end{picture}} \\
 & = & 2\raisebox{-4pt}{\begin{picture}(64,15)(-2,0)
	 \put(0, 0){\framebox(15,15){$1$}}
	 \put(15, 0){\framebox(15,15){$2$}}
	 \put(30, 0){\framebox(15,15){$3$}}
	 \put(45, 0){\framebox(15,15){$4$}}
	 \end{picture}}.
\end{eqnarray*}

Similarly, we have
\begin{eqnarray*}
 \partial_1\left(\raisebox{-12pt}{\begin{picture}(49,30)(-2,0)
	 \put(0, 0){\framebox(15,15){$1$}}
	 \put(15, 0){\framebox(15,15){$2$}}
	 \put(15, 15){\framebox(15,15){$3$}}
	 \put(30, 0){\framebox(15,15){$4$}}
	\end{picture}} 
	\right) & = & 2\raisebox{-4pt}{\begin{picture}(64,15)(-2,0)
	 \put(0, 0){\framebox(15,15){$1$}}
	 \put(15, 0){\framebox(15,15){$2$}}
	 \put(30, 0){\framebox(15,15){$3$}}
	 \put(45, 0){\framebox(15,15){$4$}}
	 \end{picture}} \\
 \partial_1\left(\raisebox{-12pt}{\begin{picture}(49,30)(-2,0)
	 \put(0, 0){\framebox(15,15){$1$}}
	 \put(15, 0){\framebox(15,15){$2$}}
	 \put(30, 0){\framebox(15,15){$3$}}
	 \put(30, 15){\framebox(15,15){$4$}}
	\end{picture}}\right)
 & = & 2\raisebox{-4pt}{\begin{picture}(64,15)(-2,0)
	 \put(0, 0){\framebox(15,15){$1$}}
	 \put(15, 0){\framebox(15,15){$2$}}
	 \put(30, 0){\framebox(15,15){$3$}}
	 \put(45, 0){\framebox(15,15){$4$}}
			\end{picture}}.
\end{eqnarray*}

For $2$-cells, we have
\begin{eqnarray*}
 \partial_2\left(\raisebox{-20pt}{\begin{picture}(34,45)(-2,0)
	 \put(0, 0){\framebox(15,15){$1$}}
	 \put(15, 0){\framebox(15,15){$2$}}
	 \put(15, 15){\framebox(15,15){$3$}}
	 \put(15, 30){\framebox(15,15){$4$}}
	\end{picture}}\right)
 & = & 	 3\raisebox{-12pt}{\begin{picture}(49,30)(-2,0)
	 \put(0, 0){\framebox(15,15){$1$}}
	 \put(15, 0){\framebox(15,15){$2$}}
	 \put(30, 0){\framebox(15,15){$3$}}
	 \put(30, 15){\framebox(15,15){$4$}}
	\end{picture}}
	-3\raisebox{-12pt}{\begin{picture}(49,30)(-2,0)
	 \put(0, 0){\framebox(15,15){$1$}}
	 \put(15, 0){\framebox(15,15){$2$}}
	 \put(15, 15){\framebox(15,15){$3$}}
	 \put(30, 0){\framebox(15,15){$4$}}
	\end{picture}} \\
 \partial_2\left(\raisebox{-12pt}{\begin{picture}(34,30)(-2,0)
	 \put(0, 0){\framebox(15,15){$1$}}
	 \put(0, 15){\framebox(15,15){$2$}}
	 \put(15, 0){\framebox(15,15){$3$}}
	 \put(15, 15){\framebox(15,15){$4$}}
	\end{picture}}\right) 
 & = &  2\raisebox{-12pt}{\begin{picture}(49,30)(-2,0)
	 \put(0, 0){\framebox(15,15){$1$}}
	 \put(15, 0){\framebox(15,15){$2$}}
	 \put(30, 0){\framebox(15,15){$3$}}
	 \put(30, 15){\framebox(15,15){$4$}}
	\end{picture}}
	-2\raisebox{-12pt}{\begin{picture}(49,30)(-2,0)
	 \put(0, 0){\framebox(15,15){$1$}}
	 \put(0, 15){\framebox(15,15){$2$}}
	 \put(15, 0){\framebox(15,15){$3$}}
	 \put(30, 0){\framebox(15,15){$4$}}
			   \end{picture}} \\
 \partial_2\left(\raisebox{-20pt}{\begin{picture}(34,45)(-2,0)
	 \put(0, 0){\framebox(15,15){$1$}}
	 \put(0, 15){\framebox(15,15){$2$}}
	 \put(0, 30){\framebox(15,15){$3$}}
	 \put(15, 0){\framebox(15,15){$4$}}
	\end{picture}}\right)
 & = & 3\raisebox{-12pt}{\begin{picture}(49,30)(-2,0)
	 \put(0, 0){\framebox(15,15){$1$}}
	 \put(15, 0){\framebox(15,15){$2$}}
	 \put(15, 15){\framebox(15,15){$3$}}
	 \put(30, 0){\framebox(15,15){$4$}}
	\end{picture}} 
	- 3\raisebox{-12pt}{\begin{picture}(49,30)(-2,0)
	 \put(0, 0){\framebox(15,15){$1$}}
	 \put(0, 15){\framebox(15,15){$2$}}
	 \put(15, 0){\framebox(15,15){$3$}}
	 \put(30, 0){\framebox(15,15){$4$}}
	\end{picture}}.
\end{eqnarray*}

Finally, for $3$-cells, we have
\begin{eqnarray*}
 \partial_3\left(\raisebox{-28pt}{\begin{picture}(19,60)(-2,0)
	 \put(0, 0){\framebox(15,15){$1$}}
	 \put(0, 15){\framebox(15,15){$2$}}
	 \put(0, 30){\framebox(15,15){$3$}}
	 \put(0, 45){\framebox(15,15){$4$}}
	 \end{picture}}\right)
 & = & 4\raisebox{-20pt}{\begin{picture}(34,45)(-2,0)
	 \put(0, 0){\framebox(15,15){$1$}}
	 \put(15, 0){\framebox(15,15){$2$}}
	 \put(15, 15){\framebox(15,15){$3$}}
	 \put(15, 30){\framebox(15,15){$4$}}
	\end{picture}} 
	- 6 \raisebox{-12pt}{\begin{picture}(34,30)(-2,0)
	 \put(0, 0){\framebox(15,15){$1$}}
	 \put(0, 15){\framebox(15,15){$2$}}
	 \put(15, 0){\framebox(15,15){$3$}}
	 \put(15, 15){\framebox(15,15){$4$}}
	\end{picture}} 
	+ 4 \raisebox{-20pt}{\begin{picture}(34,45)(-2,0)
	 \put(0, 0){\framebox(15,15){$1$}}
	 \put(0, 15){\framebox(15,15){$2$}}
	 \put(0, 30){\framebox(15,15){$3$}}
	 \put(15, 0){\framebox(15,15){$4$}}
			   \end{picture}}.
\end{eqnarray*}

Thus we have the following well-known result.

\begin{proposition}
 When $p>3$, we have
 \[
  H_i(C_*(\Sal(\mathcal{A}_3))\otimes_{\Sigma_4}\F_p(\pm 1)) = 0
 \]
 for all $i$.
\end{proposition}

\begin{proof}
 Since $p\neq 2$,
 \begin{eqnarray*}
  \Ima\partial_1 & = & \left\langle
 2\raisebox{-6pt}{\begin{picture}(64,15)(-2,0) 
	 \put(0, 0){\framebox(15,15){$1$}}
	 \put(15, 0){\framebox(15,15){$2$}}
	 \put(30, 0){\framebox(15,15){$3$}}
	 \put(45, 0){\framebox(15,15){$4$}}
		  \end{picture}}\right\rangle =
\left\langle \raisebox{-6pt}{\begin{picture}(64,15)(-2,0)
	 \put(0, 0){\framebox(15,15){$1$}}
	 \put(15, 0){\framebox(15,15){$2$}}
	 \put(30, 0){\framebox(15,15){$3$}}
	 \put(45, 0){\framebox(15,15){$4$}}
	 \end{picture}}\right\rangle \\
  \Ker\partial_1 & = &
   \left\langle
    \raisebox{-12pt}{\begin{picture}(49,30)(-2,0)
		      \put(0, 0){\framebox(15,15){$1$}}
		      \put(15, 0){\framebox(15,15){$2$}}
		      \put(15, 15){\framebox(15,15){$3$}}
		      \put(30, 0){\framebox(15,15){$4$}}
		     \end{picture}} 
    - \raisebox{-12pt}{\begin{picture}(49,30)(-2,0)
			\put(0, 0){\framebox(15,15){$1$}}
			\put(15, 0){\framebox(15,15){$2$}}
			\put(30, 0){\framebox(15,15){$3$}}
			\put(30, 15){\framebox(15,15){$4$}}
		       \end{picture}},
 \raisebox{-12pt}{\begin{picture}(49,30)(-2,0)
	 \put(0, 0){\framebox(15,15){$1$}}
	 \put(15, 0){\framebox(15,15){$2$}}
	 \put(15, 15){\framebox(15,15){$3$}}
	 \put(30, 0){\framebox(15,15){$4$}}
	\end{picture}}
	- \raisebox{-12pt}{\begin{picture}(49,30)(-2,0)
	 \put(0, 0){\framebox(15,15){$1$}}
	 \put(0, 15){\framebox(15,15){$2$}}
	 \put(15, 0){\framebox(15,15){$3$}}
	 \put(30, 0){\framebox(15,15){$4$}}
	\end{picture}}\right\rangle.
 \end{eqnarray*}

 Thus we have
 \[
 H_0(C_*(\Sal(\mathcal{C}_3^2))\otimes_{\Sigma_4}\F_p(\pm 1)) = 0. 
 \]

 Since $p\neq 3$, we have
 \[
 \Ima\partial_2 = 
\left\langle
    \raisebox{-12pt}{\begin{picture}(49,30)(-2,0)
		      \put(0, 0){\framebox(15,15){$1$}}
		      \put(15, 0){\framebox(15,15){$2$}}
		      \put(15, 15){\framebox(15,15){$3$}}
		      \put(30, 0){\framebox(15,15){$4$}}
		     \end{picture}} 
    - \raisebox{-12pt}{\begin{picture}(49,30)(-2,0)
			\put(0, 0){\framebox(15,15){$1$}}
			\put(15, 0){\framebox(15,15){$2$}}
			\put(30, 0){\framebox(15,15){$3$}}
			\put(30, 15){\framebox(15,15){$4$}}
		       \end{picture}},
 \raisebox{-12pt}{\begin{picture}(49,30)(-2,0)
	 \put(0, 0){\framebox(15,15){$1$}}
	 \put(15, 0){\framebox(15,15){$2$}}
	 \put(15, 15){\framebox(15,15){$3$}}
	 \put(30, 0){\framebox(15,15){$4$}}
	\end{picture}}
	- \raisebox{-12pt}{\begin{picture}(49,30)(-2,0)
	 \put(0, 0){\framebox(15,15){$1$}}
	 \put(0, 15){\framebox(15,15){$2$}}
	 \put(15, 0){\framebox(15,15){$3$}}
	 \put(30, 0){\framebox(15,15){$4$}}
	\end{picture}}\right\rangle = \Ker\partial_1
 \]
 and
 \[
 H_1(C_*(\Sal(\mathcal{C}_3^2))\otimes_{\Sigma_4}\F_p(\pm 1)) = 0. 
 \]
We also have
 \[
  \Ker\partial_2 =
   \left\langle
    \frac{1}{3}\raisebox{-20pt}{\begin{picture}(34,45)(-2,0) 
	 \put(0, 0){\framebox(15,15){$1$}}
	 \put(15, 0){\framebox(15,15){$2$}}
	 \put(15, 15){\framebox(15,15){$3$}}
	 \put(15, 30){\framebox(15,15){$4$}}
				\end{picture}} 
	- \frac{1}{2} \raisebox{-12pt}{\begin{picture}(34,30)(-2,0)
	 \put(0, 0){\framebox(15,15){$1$}}
	 \put(0, 15){\framebox(15,15){$2$}}
	 \put(15, 0){\framebox(15,15){$3$}}
	 \put(15, 15){\framebox(15,15){$4$}}
	\end{picture}} 
	+ \frac{1}{3} \raisebox{-20pt}{\begin{picture}(34,45)(-2,0)
	 \put(0, 0){\framebox(15,15){$1$}}
	 \put(0, 15){\framebox(15,15){$2$}}
	 \put(0, 30){\framebox(15,15){$3$}}
	 \put(15, 0){\framebox(15,15){$4$}}
			   \end{picture}}\right\rangle 
 = \Ima \partial_3
 \]
 and
 \[
 H_2(C_*(\Sal(\mathcal{C}_3^2))\otimes_{\Sigma_4}\F_p(\pm 1)) = 0. 
 \]
 Since $\Ker\partial_3=0$, 
 \[
 H_3(C_*(\Sal(\mathcal{C}_3^2))\otimes_{\Sigma_4}\F_p(\pm 1)) = 0.     
 \]
\end{proof}

\begin{proposition}
 When $p=3$, we have
 \begin{eqnarray*}
  H_0(C_*(\Sal(\mathcal{C}_3^2))\otimes_{\Sigma_4}\F_3(\pm 1)) & = & 0 \\
  H_1(C_*(\Sal(\mathcal{C}_3^2))\otimes_{\Sigma_4}\F_3(\pm 1)) & = & 
   \left\langle
    \left[\raisebox{-12pt}{\begin{picture}(49,30)(-2,0)
	 \put(0, 0){\framebox(15,15){$1$}}
	 \put(15, 0){\framebox(15,15){$2$}}
	 \put(15, 15){\framebox(15,15){$3$}}
	 \put(30, 0){\framebox(15,15){$4$}}
		     \end{picture}} 
	- \raisebox{-12pt}{\begin{picture}(49,30)(-2,0)
	 \put(0, 0){\framebox(15,15){$1$}}
	 \put(15, 0){\framebox(15,15){$2$}}
	 \put(30, 0){\framebox(15,15){$3$}}
	 \put(30, 15){\framebox(15,15){$4$}}
	\end{picture}}\right]
   \right\rangle \cong \F_3 \\
  H_2(C_*(\Sal(\mathcal{C}_3^2))\otimes_{\Sigma_4}\F_3(\pm 1)) & = & 
   \left\langle
    \left[\raisebox{-20pt}{\begin{picture}(34,45)(-2,0)
	 \put(0, 0){\framebox(15,15){$1$}}
	 \put(15, 0){\framebox(15,15){$2$}}
	 \put(15, 15){\framebox(15,15){$3$}}
	 \put(15, 30){\framebox(15,15){$4$}}
		     \end{picture}}- 
	\raisebox{-20pt}{\begin{picture}(34,45)(-2,0)
	 \put(0, 0){\framebox(15,15){$1$}}
	 \put(0, 15){\framebox(15,15){$2$}}
	 \put(0, 30){\framebox(15,15){$3$}}
	 \put(15, 0){\framebox(15,15){$4$}}
			   \end{picture}}\right]
   \right\rangle \cong \F_3 \\
  H_3(C_*(\Sal(\mathcal{C}_3^2))\otimes_{\Sigma_4}\F_3(\pm 1)) & = & 0.
 \end{eqnarray*}
\end{proposition}

\begin{proof}
 The differences are the computation of $H_2$ and $H_3$. The result
 follows from
 \[
  \Ker\partial_2 =
   \left\langle
    \raisebox{-20pt}{\begin{picture}(34,45)(-2,0)
	 \put(0, 0){\framebox(15,15){$1$}}
	 \put(15, 0){\framebox(15,15){$2$}}
	 \put(15, 15){\framebox(15,15){$3$}}
	 \put(15, 30){\framebox(15,15){$4$}}
		     \end{picture}}, 
	\raisebox{-20pt}{\begin{picture}(34,45)(-2,0)
	 \put(0, 0){\framebox(15,15){$1$}}
	 \put(0, 15){\framebox(15,15){$2$}}
	 \put(0, 30){\framebox(15,15){$3$}}
	 \put(15, 0){\framebox(15,15){$4$}}
			   \end{picture}}
   \right\rangle.
 \]
 The details are omitted.
\end{proof}

\begin{remark}
 It is well-known that
 \[
  H_*(\Omega^2S^{2n+1};\F_3) \cong \Lambda(Q_1^a(x_{2n-1}) \mid a\ge 0)
 \otimes \F_3[\beta Q_1^{a+1}(x) \mid a\ge 0].
 \]
 Under the Snaith splitting (and dimension shifts), the generators in
 $H_1(C_*(\Sal(\mathcal{C}_3^2))\otimes_{\Sigma_4}\F_3(\pm 1))$ and 
 $H_2(C_*(\Sal(\mathcal{C}_3^2))\otimes_{\Sigma_4}\F_3(\pm 1))$
 correspond to $x_{2n-1}\beta Q_1(x_{2n-1})$ and
 $x_{2n-1}Q_1(x_{2n-1})$, respectively. Since
 \begin{eqnarray*}
  \deg x_{2n-1}\beta Q_1(x_{2n-1}) & = & (2n-1)+3(2n-1)+(p-2) \\
  & = & 4(2n-1) + 1 \\
  \deg x_{2n-1}Q_1(x_{2n-1}) & = & (2n-1)+3(2n-1)+(p-1) \\
  & = & 4(2n-1) + 2.
 \end{eqnarray*}
\end{remark}

The $2$-primary case is simpler, since we don't have to worry
about the signs. We have
\begin{eqnarray*}
 H_*(C_*(\Sal(\mathcal{A}_3))\otimes_{\Sigma_4}\F_2(\pm 1)) & = &
 H_*(C_*(\Sal(\mathcal{A}_3))\otimes_{\Sigma_4}\F_2) \\
 & = &
 H_*(C_*(\Sal(\mathcal{A}_3)/\Sigma_4)\otimes\F_2) \\
 & = &
 H_*(F(\C,4)/\Sigma_4;\F_2).
\end{eqnarray*}

We obtain the following well-known result by elementary
calculations. Details are omitted.

\begin{proposition}
 The homology $H_*(F(\C,4)/\Sigma_4;\F_2)$ has the following
 description:
 \begin{eqnarray*}
  H_0(F(\C,4)/\Sigma_4;\F_2) & = & \left\langle 
	 \raisebox{-6pt}{\begin{picture}(64,15)(-2,0)
	 \put(0, 0){\framebox(15,15){$1$}}
	 \put(15, 0){\framebox(15,15){$2$}}
	 \put(30, 0){\framebox(15,15){$3$}}
	 \put(45, 0){\framebox(15,15){$4$}}
			 \end{picture}} 
			 \right\rangle \\
  & \cong & \F_2 \\
  H_1(F(\C,4)/\Sigma_4;\F_2)
  & = & \left\langle\ 
	\raisebox{-12pt}{\begin{picture}(49,30)(-2,0)
	 \put(0, 0){\framebox(15,15){$1$}}
	 \put(0, 15){\framebox(15,15){$2$}}
	 \put(15, 0){\framebox(15,15){$3$}}
	 \put(30, 0){\framebox(15,15){$4$}}
	\end{picture}} = 
	\raisebox{-12pt}{\begin{picture}(49,30)(-2,0)
	 \put(0, 0){\framebox(15,15){$1$}}
	 \put(15, 0){\framebox(15,15){$2$}}
	 \put(15, 15){\framebox(15,15){$3$}}
	 \put(30, 0){\framebox(15,15){$4$}}
	\end{picture}} = 
	\raisebox{-12pt}{\begin{picture}(49,30)(-2,0)
	 \put(0, 0){\framebox(15,15){$1$}}
	 \put(15, 0){\framebox(15,15){$2$}}
	 \put(30, 0){\framebox(15,15){$3$}}
	 \put(30, 15){\framebox(15,15){$4$}}
	\end{picture}} 
	\right\rangle \\
  & \cong & \F_2 \\
  H_2(F(\C,4)/\Sigma_4;\F_2) & = & 
   \left\langle 
	\raisebox{-12pt}{\begin{picture}(34,30)(-2,0)
	 \put(0, 0){\framebox(15,15){$1$}}
	 \put(0, 15){\framebox(15,15){$2$}}
	 \put(15, 0){\framebox(15,15){$3$}}
	 \put(15, 15){\framebox(15,15){$4$}}
	\end{picture}}   
	\right\rangle \\
  & \cong & \F_2 \\
  H_3(F(\C,4)/\Sigma_4;\F_2) 
  & = & \left\langle
	 \raisebox{-28pt}{\begin{picture}(19,60)(-2,0)
	 \put(0, 0){\framebox(15,15){$1$}}
	 \put(0, 15){\framebox(15,15){$2$}}
	 \put(0, 30){\framebox(15,15){$3$}}
	 \put(0, 45){\framebox(15,15){$4$}}
	 \end{picture}} 
	\right\rangle \\
  & \cong & \F_2
 \end{eqnarray*}
\end{proposition}

\begin{remark}
 Under the stable splitting
 \[
  \Omega^2S^{n+2} \shot \bigvee_j
 F(\C,j)_{+}\wedge_{\Sigma_j}(S^n)^{\wedge j},
 \]
 the elements in the mod $2$ homology of $F(\C,4)/\Sigma_4$, up to a
 shift of degree, correspond to elements in $H_*(\Omega^2S^{n+2})$ as
 follows:
 \begin{eqnarray*}
  Q_0^2(x) & \longleftrightarrow & 
   	 \raisebox{-6pt}{\begin{picture}(64,15)(-2,0)
	 \put(0, 0){\framebox(15,15){$1$}}
	 \put(15, 0){\framebox(15,15){$2$}}
	 \put(30, 0){\framebox(15,15){$3$}}
	 \put(45, 0){\framebox(15,15){$4$}}
			 \end{picture}} \\
  Q_0(x)Q_1(x) & \longleftrightarrow & 
   \raisebox{-12pt}{\begin{picture}(49,30)(-2,0)
	 \put(0, 0){\framebox(15,15){$1$}}
	 \put(0, 15){\framebox(15,15){$2$}}
	 \put(15, 0){\framebox(15,15){$3$}}
	 \put(30, 0){\framebox(15,15){$4$}}
	\end{picture}} = 
	\raisebox{-12pt}{\begin{picture}(49,30)(-2,0)
	 \put(0, 0){\framebox(15,15){$1$}}
	 \put(15, 0){\framebox(15,15){$2$}}
	 \put(15, 15){\framebox(15,15){$3$}}
	 \put(30, 0){\framebox(15,15){$4$}}
	\end{picture}} = 
	\raisebox{-12pt}{\begin{picture}(49,30)(-2,0)
	 \put(0, 0){\framebox(15,15){$1$}}
	 \put(15, 0){\framebox(15,15){$2$}}
	 \put(30, 0){\framebox(15,15){$3$}}
	 \put(30, 15){\framebox(15,15){$4$}}
	\end{picture}} \\
  Q_0Q_1(x) & \longleftrightarrow & 
	\raisebox{-12pt}{\begin{picture}(34,30)(-2,0)
	 \put(0, 0){\framebox(15,15){$1$}}
	 \put(0, 15){\framebox(15,15){$2$}}
	 \put(15, 0){\framebox(15,15){$3$}}
	 \put(15, 15){\framebox(15,15){$4$}}
	\end{picture}} \\
  Q_1^2(x) & \longleftrightarrow &
   \raisebox{-28pt}{\begin{picture}(19,60)(-2,0)
		     \put(0, 0){\framebox(15,15){$1$}}
		     \put(0, 15){\framebox(15,15){$2$}}
		     \put(0, 30){\framebox(15,15){$3$}}
		     \put(0, 45){\framebox(15,15){$4$}}
		    \end{picture}}
 \end{eqnarray*}
\end{remark}

\section{The Center of Mass Configuration}
\label{center}

Let us recall the definition of the center of mass configuration space
introduced by F.~Cohen and Kamiyama in \cite{math.AT/0611732}.

\begin{definition}
 For $I, J \subset \{1, \cdots, n\}$, define
 \[
 L_{I,J} = \rset{(x_1,\cdots,x_n)\in\R^n}{ |J|\sum_{i\in I}x_i =
 |I|\sum_{j\in J}x_j }
 \]
 For $\ell<n$, define a real central hyperplane arrangement
 $\mathcal{C}_{n-1}^{\ell}$ by
 \[
 \mathcal{C}_{n-1}^{\ell} = \lset{L_{I,J}}{ I,J\subset \{1,\cdots,n\},
 |I|=|J|=p, I \neq J}.
 \]
 The configuration space of $n$ points with distinct center of mass of
 $\ell$ points is defined as the complement of the complexification of
 $\mathcal{C}_{n-1}^{\ell}$
 \[
  M_{\ell}(\C,n) = \C^n - \bigcup_{L_{I,J}\in \mathcal{C}_{n-1}^{\ell}}
 L_{I,J}\otimes\C 
 \]
\end{definition}

As we have seen in \S\ref{introduction}, we have the following
inclusions of arrangements 
\[
 \mathcal{A}_{n-1} = \mathcal{C}_{n-1}^1 \subset \mathcal{C}_{n-1}^2
 \subset \cdots \subset \mathcal{C}_{n-1}^{[\frac{n}{2}]} \supset \cdots
 \supset \mathcal{C}_{n-1}^{n-1} = \mathcal{A}_{n-1}.
\]
By Proposition \ref{Salvetti_by_poset}, we obtain a sequence of maps
between the Salvetti complexes 
\[
 \Sal(\mathcal{C}_{n-1}^{p}) \to \cdots \to \Sal(\mathcal{C}_{n-1}^2) \to 
 \Sal(\mathcal{C}_{n-1}^1) = \Sal(\mathcal{A}_{n-1})
\]
when $p\le [\frac{n}{2}]$. We also have
\[
 \Sal(\mathcal{C}_{n-1}^{p}) = \Sal(\mathcal{C}_{n-1}^{n-p}) \to \cdots
 \to \Sal(\mathcal{C}_{n-1}^2) \to  
 \Sal(\mathcal{C}_{n-1}^1) = \Sal(\mathcal{A}_{n-1})
\]
when $p>[\frac{n}{2}]$.

We would like to know if these inclusions induce isomorphisms of
homology groups with coefficients in $\F_p(\pm 1)$.
Our strategy is to compute the homology of the kernel of the map 
\[
 i^k_n : C_*(\Sal(\mathcal{C}_{n-1}^k))\otimes_{\Sigma_n}\F_p(\pm 1)
 \longrightarrow
 C_*(\Sal(\mathcal{C}_{n-1}^{k-1}))\otimes_{\Sigma_n}\F_{p}(\pm 1),
\]
for $k\le [\frac{n}{2}]$. By Lemma \ref{i_is_surjective}, these chain
maps are surjective.

\begin{corollary}
 \label{i^k_n_is_surjective}
 For $k\le [\frac{n}{2}]$, the map
 \[
 i^k_n : C_*(\Sal(\mathcal{C}_{n-1}^k))\otimes_{\Sigma_n}\F_p(\pm 1)
 \longrightarrow
 C_*(\Sal(\mathcal{C}_{n-1}^{k-1}))\otimes_{\Sigma_n}\F_{p}(\pm 1),
 \]
 is surjective. 
\end{corollary}

In the rest of this article, we consider the first stage, i.e.
\[
 i^2_n : C_*(\Sal(\mathcal{C}_{n-1}^2))\otimes_{\Sigma_n}\F_p(\pm 1)
 \longrightarrow
 C_*(\Sal(\mathcal{C}_{n-1}^{1}))\otimes_{\Sigma_n}\F_{p}(\pm 1) = 
 C_*(\Sal(\mathcal{A}_{n-1}))\otimes_{\Sigma_n}\F_{p}(\pm 1).
\]
When $n=3$,
\[
  \mathcal{C}^2_2 = \mathcal{C}_2^{3-2} = \mathcal{C}_2^1 =
  \mathcal{A}_2 
\]
and there is nothing to compute. The first nontrivial case is
\[
 i^2_4 : C_*(\Sal(\mathcal{C}_{3}^2))\otimes_{\Sigma_4}\F_p(\pm 1)
 \longrightarrow
 C_*(\Sal(\mathcal{C}_{3}^{1}))\otimes_{\Sigma_4}\F_{p}(\pm 1) = 
 C_*(\Sal(\mathcal{A}_{3}))\otimes_{\Sigma_4}\F_{p}(\pm 1).
\]

By Corollary \ref{i^k_n_is_surjective}, it suffices to calculate
the kernel of $i_4^2$ in order to compare
$H_*(C_*(\Sal(\mathcal{C}_{3}^2))\otimes_{\Sigma_4}\F_p(\pm 1))$ and
$H_*(C_*(\Sal(\mathcal{A}_3))\otimes_{\Sigma_4}\F_p(\pm 1))$. 

\begin{definition}
 \label{abbreviated_chain_complexes}
 We denote
 \[
  K^{n,k}_* = \Ker (i_{n}^k :
 C_*(\Sal(\mathcal{C}_{n-1}^k))\otimes_{\Sigma_n}\F_p(\pm 1) \to
 C_*(\Sal(\mathcal{C}_{n-1}^{k-1}))\otimes_{\Sigma_n}\F_p(\pm 1)).
 \]
 For simplicity, we also abbreviate
 \[
  C_*^{n,k} = C_*(\Sal(\mathcal{C}_{n-1}^k))\otimes_{\Sigma_n}\F_p(\pm 1).
 \]
\end{definition}

Thus we have a short exact sequence of chain complexes
\[
 0 \longrightarrow K_*^{n,k} \longrightarrow C_*^{n,k} \rarrow{i_n^k}
 C_*^{n,k-1} \longrightarrow 0.
\]

\subsection{The Face Lattice of $\mathcal{C}_3^2$}
\label{face_lattice}

In order to compute
$H_*(\Sal(\mathcal{C}_3^2)\otimes_{\Sigma_4}\F_p(\pm 1))$, the first
step is to determine the face lattice of 
$\mathcal{C}_3^2$. 

Since $\mathcal{C}_3^2 = \mathcal{A}_3 \cup \{L_{\{1,2\},\{3,4\}},
L_{\{1,3\},\{2,4\}}, L_{\{1,4\},\{2,3\}}\}$, the faces of
$\mathcal{C}_3^2$ are given by splitting the faces of $\mathcal{A}_3$ by
the hyperplanes
\[
 L_{\{1,2\},\{3,4\}}, L_{\{1,3\},\{2,4\}}, L_{\{1,4\},\{2,3\}}.
\]
In order to understand these cuttings, let us see how the chamber 
\[
 \{(x_1,x_2,x_3,x_4)\in\R^4 \mid x_1 < x_2 < x_3 < x_4\}
\]
 is cut. Notice that under the action of $\Sigma_4$, the cells in
$M_2(\C,4)$ can be represented by cells related to this chamber of
$\mathcal{A}_3$.

The only hyperplane among $L_{\{1,2\},\{3,4\}}, L_{\{1,3\},\{2,4\}},
L_{\{1,4\},\{2,3\}}$ that intersects with this chamber is
$L_{\{1,4\},\{2,3\}}$ and the chamber is cut into two pieces:
\begin{multline*}
 \{(x_1,x_2,x_3,x_4)\in\R^4 \mid x_1 < x_2 < x_3 < x_4\} \\
 = \{(x_1,x_2,x_3,x_4)\in\R^4 \mid x_1 < x_2 < x_3 <x_4, x_1+x_4\le
 x_2+x_3\} \\
 \cup \{(x_1,x_2,x_3,x_4)\in\R^4 \mid x_1 < x_2 < x_3 < x_4,
 x_1+x_4\ge x_2+x_3\}.
\end{multline*}
We denote these chambers by the following symbols:
\begin{multline*}
 \raisebox{-6pt}{\begin{picture}(75,15)(0,0)
	 \put(0, 0){\framebox(15,15){$1$}}
	 \put(15, 0){\framebox(15,15){}}
	 \put(30, 0){\framebox(15,15){$2$}}
	 \put(45, 0){\framebox(15,15){$3$}}
	 \put(60, 0){\framebox(15,15){$4$}}
		 \end{picture}} \\
		  = \{(x_1,x_2,x_3,x_4)\in\R^4 \mid x_1 < x_2 <
		 x_3 < x_4, x_1+x_4<x_2+x_3\} 
\end{multline*}
\begin{multline*}
 \raisebox{-6pt}{\begin{picture}(75,15)(0,0)
	 \put(0, 0){\framebox(15,15){$1$}}
	 \put(15, 0){\framebox(15,15){$2$}}
	 \put(30, 0){\framebox(15,15){$3$}}
	 \put(45, 0){\framebox(15,15){}}
	 \put(60, 0){\framebox(15,15){$4$}}
		 \end{picture}} \\
		  = \{(x_1,x_2,x_3,x_4)\in\R^4 \mid x_1 < x_2 <
		 x_3 < x_4, x_1+x_4>x_2+x_3\}.
\end{multline*}

The faces of these chambers of $\mathcal{C}_3^2$ are also denoted by
analogous symbols. The chamber
\raisebox{-6pt}{\begin{picture}(75,15)(0,0)
		 \put(0, 0){\framebox(15,15){$1$}}
		 \put(15, 0){\framebox(15,15){}}
		 \put(30, 0){\framebox(15,15){$2$}}
		 \put(45, 0){\framebox(15,15){$3$}}
		 \put(60, 0){\framebox(15,15){$4$}}
		\end{picture}}
has five $3$-dimensional faces, but under the action of $\Sigma_4$, we
only need the following three faces:
\begin{eqnarray*}
 \raisebox{-12pt}{\begin{picture}(60,30)(0,0)
	 \put(0, 0){\framebox(15,15){$1$}}
	 \put(15, 0){\framebox(15,15){}}
	 \put(30, 0){\framebox(15,15){$2$}}
	 \put(45, 0){\framebox(15,15){$3$}}
	 \put(45, 15){\framebox(15,15){$4$}}
		 \end{picture}}
		 & = & \{(x_1,x_2,x_3,x_4)\in\R^4 \mid x_1 < x_2 <
		 x_3 = x_4\}, \\
 \raisebox{-12pt}{\begin{picture}(60,30)(0,0)
	 \put(0, 0){\framebox(15,15){$1$}}
	 \put(15, 0){\framebox(15,15){}}
	 \put(30, 0){\framebox(15,15){$2$}}
	 \put(30, 15){\framebox(15,15){$3$}}
	 \put(45, 0){\framebox(15,15){$4$}}
		 \end{picture}}
		 & = & \{(x_1,x_2,x_3,x_4)\in\R^4 \mid x_1 < x_2 =
		 x_3 < x_4, x_1+x_4<x_2+x_3\}, \\
 \raisebox{-6pt}{\begin{picture}(75,15)(0,0)
	 \put(0, 0){\framebox(15,15){$1$}}
	 \put(15, 0){\framebox(15,15){$2$}}
	 \put(30, 0){\framebox(15,15){}}
	 \put(45, 0){\framebox(15,15){$3$}}
	 \put(60, 0){\framebox(15,15){$4$}}
		 \end{picture}}
		 & = & \{(x_1,x_2,x_3,x_4)\in\R^4 \mid x_1 < x_2 <
		 x_3 < x_4, x_1+x_4=x_2+x_3\}. 
\end{eqnarray*}

Similarly, we need the following three faces for the chamber  
\raisebox{-6pt}{\begin{picture}(79,15)(-2,0)
	 \put(0, 0){\framebox(15,15){$1$}}
	 \put(15, 0){\framebox(15,15){$2$}}
	 \put(30, 0){\framebox(15,15){$3$}}
	 \put(45, 0){\framebox(15,15){}}
	 \put(60, 0){\framebox(15,15){$4$}}
		 \end{picture}}.

\begin{eqnarray*}
 \raisebox{-12pt}{\begin{picture}(60,30)(0,0)
	 \put(0, 0){\framebox(15,15){$1$}}
	 \put(0, 15){\framebox(15,15){$2$}}
	 \put(15, 0){\framebox(15,15){$3$}}
	 \put(30, 0){\framebox(15,15){}}
	 \put(45, 0){\framebox(15,15){$4$}}
		 \end{picture}}
		 & = & \{(x_1,x_2,x_3,x_4)\in\R^4 \mid x_1 = x_2 <
		 x_3 < x_4\}, \\
 \raisebox{-12pt}{\begin{picture}(60,30)(0,0)
	 \put(0, 0){\framebox(15,15){$1$}}
	 \put(15, 0){\framebox(15,15){$2$}}
	 \put(15, 15){\framebox(15,15){$3$}}
	 \put(30, 0){\framebox(15,15){}}
	 \put(45, 0){\framebox(15,15){$4$}}
		 \end{picture}}
		 & = & \{(x_1,x_2,x_3,x_4)\in\R^4 \mid x_1 < x_2 =
		 x_3 < x_4, x_1+x_4>x_2+x_3\}, \\
 \raisebox{-6pt}{\begin{picture}(75,15)(0,0)
	 \put(0, 0){\framebox(15,15){$1$}}
	 \put(15, 0){\framebox(15,15){$2$}}
	 \put(30, 0){\framebox(15,15){}}
	 \put(45, 0){\framebox(15,15){$3$}}
	 \put(60, 0){\framebox(15,15){$4$}}
		 \end{picture}}
		 & = & \{(x_1,x_2,x_3,x_4)\in\R^4 \mid x_1 < x_2 <
		 x_3 < x_4, x_1+x_4=x_2+x_3\}. 
\end{eqnarray*}

The $2$-dimensional faces we need are the following:
\begin{eqnarray*}
 \raisebox{-18pt}{\begin{picture}(45,45)(0,0)
	 \put(0, 0){\framebox(15,15){$1$}}
	 \put(15, 0){\framebox(15,15){}}
	 \put(30, 0){\framebox(15,15){$2$}}
	 \put(30, 15){\framebox(15,15){$3$}}
	 \put(30, 30){\framebox(15,15){$4$}}
		 \end{picture}}
		 & = & \{(x_1,x_2,x_3,x_4)\in\R^4 \mid x_1 < x_2 =
		 x_3 = x_4\}, \\
 \raisebox{-12pt}{\begin{picture}(45,30)(0,0)
	 \put(0, 0){\framebox(15,15){$1$}}
	 \put(0, 15){\framebox(15,15){$2$}}
	 \put(15, 0){\framebox(15,15){}}
	 \put(30, 0){\framebox(15,15){$3$}}
	 \put(30, 15){\framebox(15,15){$4$}}
		 \end{picture}}
		 & = & \{(x_1,x_2,x_3,x_4)\in\R^4 \mid x_1 = x_2 <
		 x_3 = x_4, x_1+x_4=x_2+x_3\}, \\
 \raisebox{-12pt}{\begin{picture}(100,40)(0,0)
	 \put(0, 0){\framebox(15,15){$1$}}
	 \put(15, 0){\framebox(15,15){}}
	 \put(30, 0){\framebox(15,15){$2$}}
	 \put(30, 15){\framebox(15,15){$3$}}
	 \put(45, 0){\framebox(15,15){}}
	 \put(60, 0){\framebox(15,15){$4$}}
		 \end{picture}}
		 & = & \{(x_1,x_2,x_3,x_4)\in\R^4 \mid x_1 < x_2 =
		 x_3 < x_4, x_1+x_4=x_2+x_3\}, \\
 \raisebox{-18pt}{\begin{picture}(45,45)(0,0)
	 \put(0, 0){\framebox(15,15){$1$}}
	 \put(0, 15){\framebox(15,15){$2$}}
	 \put(0, 30){\framebox(15,15){$3$}}
	 \put(15, 0){\framebox(15,15){}}
	 \put(30, 0){\framebox(15,15){$4$}}
		 \end{picture}}
		 & = & \{(x_1,x_2,x_3,x_4)\in\R^4 \mid x_1 = x_2 =
		 x_3 < x_4\}. 
\end{eqnarray*}

All these faces have the following $1$-dimensional face in common.
\[
 \raisebox{-24pt}{\begin{picture}(15,60)(0,0)
	 \put(0, 0){\framebox(15,15){$1$}}
	 \put(0, 15){\framebox(15,15){$2$}}
	 \put(0, 30){\framebox(15,15){$3$}}
	 \put(0, 45){\framebox(15,15){$4$}}
		 \end{picture}}
		  = \{(x_1,x_2,x_3,x_4)\in\R^4 \mid x_1 = x_2 =
		 x_3 = x_4\}. 
\]

Notice that $\Sigma_4$ acts on $\mathcal{L}(\mathcal{C}_3^2)$ and the
action is compatible with the ordering. We have the following
description of the poset $\mathcal{L}(\mathcal{C}_3^2)/\Sigma_4$.

\begin{lemma}
 The poset $\mathcal{L}(\mathcal{C}_3^2)/\Sigma_4$ has the following
 structure:
 
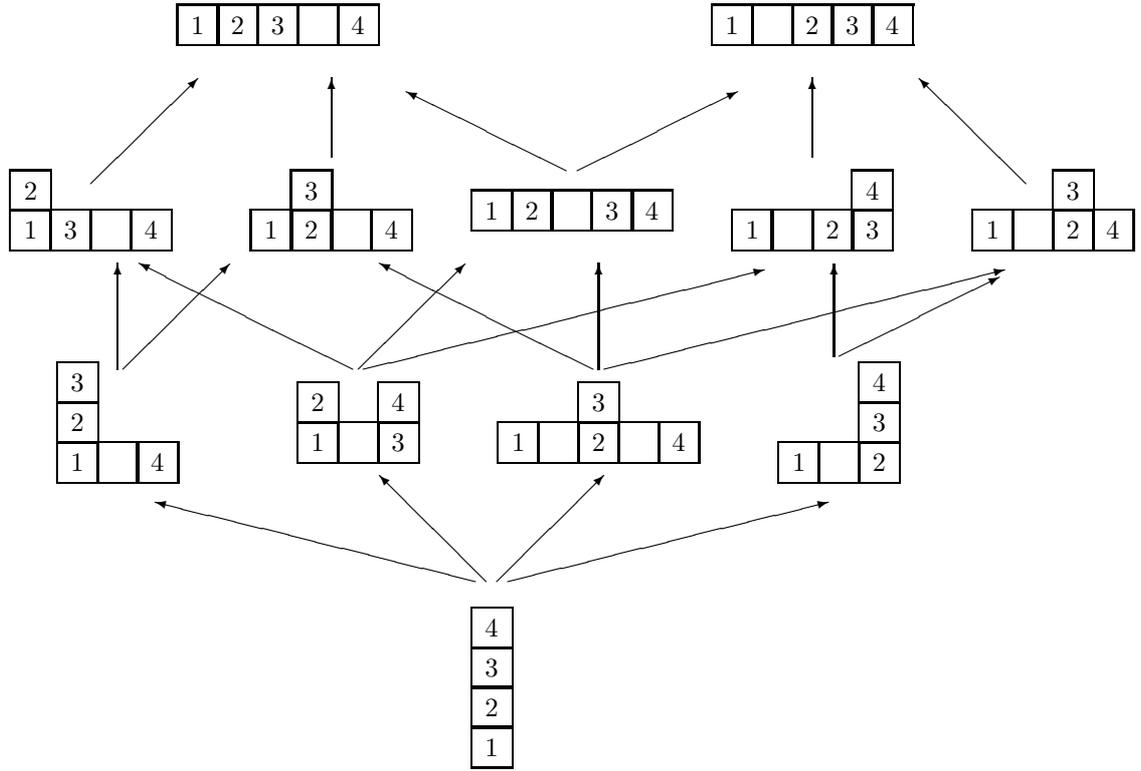
\begin{center}
  \begin{picture}(350,300)(0,0)
   \put(300,280){\makebox(0,0){\begin{picture}(75,15)(0,0)
	 \put(0, 0){\framebox(15,15){$1$}}
	 \put(15, 0){\framebox(15,15){}}
	 \put(30, 0){\framebox(15,15){$2$}}
	 \put(45, 0){\framebox(15,15){$3$}}
	 \put(60, 0){\framebox(15,15){$4$}}
		 \end{picture}}}
   \put(100,280){\makebox(0,0){\begin{picture}(75,15)(0,0)
	 \put(0, 0){\framebox(15,15){$1$}}
	 \put(15, 0){\framebox(15,15){$2$}}
	 \put(30, 0){\framebox(15,15){$3$}}
	 \put(45, 0){\framebox(15,15){}}
	 \put(60, 0){\framebox(15,15){$4$}}
		 \end{picture}}}

   \put(390,210){\makebox(0,0){\begin{picture}(60,30)(0,0)
				\put(0, 0){\framebox(15,15){$1$}}
				\put(15, 0){\framebox(15,15){}}
				\put(30, 0){\framebox(15,15){$2$}}
				\put(30, 15){\framebox(15,15){$3$}}
				\put(45, 0){\framebox(15,15){$4$}}
			       \end{picture}}}
   \put(380,220){\vector(-1,1){40}}

   \put(300,210){\makebox(0,0){\begin{picture}(60,30)(0,0)
				\put(0, 0){\framebox(15,15){$1$}}
				\put(15, 0){\framebox(15,15){}}
				\put(30, 0){\framebox(15,15){$2$}}
				\put(45, 0){\framebox(15,15){$3$}}
				\put(45, 15){\framebox(15,15){$4$}}
			       \end{picture}}}
   \put(300,230){\vector(0,1){30}}

   \put(210,210){\makebox(0,0){\begin{picture}(75,15)(0,0)
			  \put(0, 0){\framebox(15,15){$1$}}
			  \put(15, 0){\framebox(15,15){$2$}}
			  \put(30, 0){\framebox(15,15){}}
			  \put(45, 0){\framebox(15,15){$3$}}
			  \put(60, 0){\framebox(15,15){$4$}}
			 \end{picture}}}
   \put(208,225){\vector(-2,1){60}}
   \put(212,225){\vector(2,1){60}}

   \put(120,210){\makebox(0,0){\begin{picture}(60,30)(0,0)
			  \put(0, 0){\framebox(15,15){$1$}}
			  \put(15, 0){\framebox(15,15){$2$}}
			  \put(15, 15){\framebox(15,15){$3$}}
			  \put(30, 0){\framebox(15,15){}}
			  \put(45, 0){\framebox(15,15){$4$}}
			 \end{picture}}}
   \put(120,230){\vector(0,1){30}}

   \put(30,210){\makebox(0,0){\begin{picture}(60,30)(0,0)
			  \put(0, 0){\framebox(15,15){$1$}}
			  \put(0, 15){\framebox(15,15){$2$}}
			  \put(15, 0){\framebox(15,15){$3$}}
			  \put(30, 0){\framebox(15,15){}}
			  \put(45, 0){\framebox(15,15){$4$}}
		 \end{picture}}}
   \put(30,220){\vector(1,1){40}}

   \put(310,130){\makebox(0,0){\begin{picture}(45,45)(0,0)
			  \put(0, 0){\framebox(15,15){$1$}}
			  \put(15, 0){\framebox(15,15){}}
			  \put(30, 0){\framebox(15,15){$2$}}
			  \put(30, 15){\framebox(15,15){$3$}}
			  \put(30, 30){\framebox(15,15){$4$}}
			 \end{picture}}}
   \put(308,155){\vector(0,1){35}}
   \put(310,155){\vector(2,1){60}}

   \put(220,130){\makebox(0,0){\begin{picture}(75,30)(0,0)
				\put(0, 0){\framebox(15,15){$1$}}
				\put(15, 0){\framebox(15,15){}}
				\put(30, 0){\framebox(15,15){$2$}}
				\put(30, 15){\framebox(15,15){$3$}}
				\put(45, 0){\framebox(15,15){}}
				\put(60, 0){\framebox(15,15){$4$}}
			       \end{picture}}}
   \put(218,150){\vector(-2,1){80}}
   \put(220,150){\vector(0,1){40}}
   \put(222,150){\vector(4,1){150}}

   \put(130,130){\makebox(0,0){\begin{picture}(45,30)(0,0)
			       \put(0, 0){\framebox(15,15){$1$}}
			       \put(0, 15){\framebox(15,15){$2$}}
			       \put(15, 0){\framebox(15,15){}}
			       \put(30, 0){\framebox(15,15){$3$}}
			       \put(30, 15){\framebox(15,15){$4$}}
			      \end{picture}}}
   \put(128,150){\vector(-2,1){80}}
   \put(130,150){\vector(1,1){40}}
   \put(132,150){\vector(4,1){150}}

   \put(40,130){\makebox(0,0){\begin{picture}(45,45)(0,0)
			     \put(0, 0){\framebox(15,15){$1$}}
			     \put(0, 15){\framebox(15,15){$2$}}
			     \put(0, 30){\framebox(15,15){$3$}}
			     \put(15, 0){\framebox(15,15){}}
			     \put(30, 0){\framebox(15,15){$4$}}
			    \end{picture}}}
   \put(40,150){\vector(0,1){40}}
   \put(42,150){\vector(1,1){40}}

   \put(180,0){\makebox(0,0)[b]{\begin{picture}(15,60)(0,0)
				 \put(0, 0){\framebox(15,15){$1$}}
				 \put(0, 15){\framebox(15,15){$2$}}
				 \put(0, 30){\framebox(15,15){$3$}}
				 \put(0, 45){\framebox(15,15){$4$}}
				\end{picture}}}
   \put(174,70){\vector(-4,1){120}}
   \put(178,70){\vector(-1,1){40}}
   \put(182,70){\vector(1,1){40}}
   \put(186,70){\vector(4,1){120}}
  \end{picture}
 \end{center}
\end{lemma}

\subsection{The Cellular Structure on $\mathrm{Sal}(\mathcal{C}_3^2)/\Sigma_4$}
\label{cells}

Let us determine the cellular structure of
$\Sal(\mathcal{C}_3^2)/\Sigma_4$. 
The cell decomposition of the Salvetti complex for
$\mathcal{C}_3^2$ is compatible with the action of $\Sigma_4$ and the
quotient $\Sal(\mathcal{C}_3^2)/\Sigma_4$ has the induced cell
decomposition.

The cells of the Salvetti complex are labeled by pairs of a
face $F$ and a chamber $C$ with $C\ge F$. The cell for the pair $(F,C)$
is denoted by $D(F,C)$ in \S\ref{Salvetti_complex_basics}. Thus the cells of
$\Sal(\mathcal{C}_3^2)/\Sigma_4$ are in one-to-one correspondence with
elements in  
\[
 \{([F],[C]) \mid [F]\in\mathcal{L}(\mathcal{C}_3^2)/\Sigma_4, [C] \in
 \mathcal{L}^{(0)}(\mathcal{C}_3^2)/\Sigma_4, F\le C\}.
\]

In the case of $\mathcal{C}_3^2$, there are only two chambers in 
$\mathcal{L}^{(0)}(\mathcal{C}_3^2)/\Sigma_4$, and we denote the cells
corresponding to the pair $\left([F],\left[
\raisebox{-6pt}{\begin{picture}(79,15)(-2,0)
	 \put(0, 0){\framebox(15,15){$1$}}
	 \put(15, 0){\framebox(15,15){$2$}}
	 \put(30, 0){\framebox(15,15){$3$}}
	 \put(45, 0){\framebox(15,15){}}
	 \put(60, 0){\framebox(15,15){$4$}}
		\end{picture}}\right]\right)$ and
$\left([F],\left[\raisebox{-6pt}{\begin{picture}(79,15)(-2,0)
	 \put(0, 0){\framebox(15,15){$1$}}
	 \put(15, 0){\framebox(15,15){}}
	 \put(30, 0){\framebox(15,15){$2$}}
	 \put(45, 0){\framebox(15,15){$3$}}
	 \put(60, 0){\framebox(15,15){$4$}}
		   \end{picture}}\right]\right)$ by $F^+$ and
	 $F^-$, respectively.
To be more efficient, we simply denote them by $F$ when $F$ is contained
in only one chamber.

More explicitly,

\begin{lemma}
 $\Sal(\mathcal{C}_3^2)/\Sigma_4$
 has 
 \begin{itemize}
  \item  two $0$-cells
	 \begin{center}
	  \raisebox{-6pt}{\begin{picture}(79,15)(-2,0)
	 \put(0, 0){\framebox(15,15){$1$}}
	 \put(15, 0){\framebox(15,15){$2$}}
	 \put(30, 0){\framebox(15,15){$3$}}
	 \put(45, 0){\framebox(15,15){}}
	 \put(60, 0){\framebox(15,15){$4$}}
		\end{picture}}, 
	  \raisebox{-6pt}{\begin{picture}(79,15)(-2,0)
	 \put(0, 0){\framebox(15,15){$1$}}
	 \put(15, 0){\framebox(15,15){}}
	 \put(30, 0){\framebox(15,15){$2$}}
	 \put(45, 0){\framebox(15,15){$3$}}
	 \put(60, 0){\framebox(15,15){$4$}}
		   \end{picture}}.
	 \end{center}

  \item  six $1$-cells 
	 \begin{center}
	  \raisebox{-12pt}{\begin{picture}(64,30)(-2,0)
	   \put(0, 0){\framebox(15,15){$1$}}
	   \put(15, 0){\framebox(15,15){}}
	   \put(30, 0){\framebox(15,15){$2$}}
	   \put(30, 15){\framebox(15,15){$3$}}
	   \put(45, 0){\framebox(15,15){$4$}}
	  \end{picture}},
	  \raisebox{-12pt}{\begin{picture}(64,30)(-2,0)
				\put(0, 0){\framebox(15,15){$1$}}
				\put(15, 0){\framebox(15,15){}}
				\put(30, 0){\framebox(15,15){$2$}}
				\put(45, 0){\framebox(15,15){$3$}}
				\put(45, 15){\framebox(15,15){$4$}}
			   \end{picture}},
	  $\raisebox{-6pt}{\begin{picture}(79,15)(-2,0)
			  \put(0, 0){\framebox(15,15){$1$}}
			  \put(15, 0){\framebox(15,15){$2$}}
			  \put(30, 0){\framebox(15,15){}}
			  \put(45, 0){\framebox(15,15){$3$}}
			  \put(60, 0){\framebox(15,15){$4$}}
			 \end{picture}}^+$,
	  $\raisebox{-6pt}{\begin{picture}(79,15)(-2,0)
			  \put(0, 0){\framebox(15,15){$1$}}
			  \put(15, 0){\framebox(15,15){$2$}}
			  \put(30, 0){\framebox(15,15){}}
			  \put(45, 0){\framebox(15,15){$3$}}
			  \put(60, 0){\framebox(15,15){$4$}}
			 \end{picture}}^-$,
	  \raisebox{-12pt}{\begin{picture}(64,30)(-2,0)
			  \put(0, 0){\framebox(15,15){$1$}}
			  \put(15, 0){\framebox(15,15){$2$}}
			  \put(15, 15){\framebox(15,15){$3$}}
			  \put(30, 0){\framebox(15,15){}}
			  \put(45, 0){\framebox(15,15){$4$}}
			   \end{picture}},
	  \raisebox{-12pt}{\begin{picture}(64,30)(-2,0)
			  \put(0, 0){\framebox(15,15){$1$}}
			  \put(0, 15){\framebox(15,15){$2$}}
			  \put(15, 0){\framebox(15,15){$3$}}
			  \put(30, 0){\framebox(15,15){}}
			  \put(45, 0){\framebox(15,15){$4$}}
			   \end{picture}}.
	 \end{center}

  \item  six $2$-cells
	 \begin{center}
	  \raisebox{-18pt}{\begin{picture}(49,45)(-2,0)
			  \put(0, 0){\framebox(15,15){$1$}}
			  \put(15, 0){\framebox(15,15){}}
			  \put(30, 0){\framebox(15,15){$2$}}
			  \put(30, 15){\framebox(15,15){$3$}}
			  \put(30, 30){\framebox(15,15){$4$}}
			   \end{picture}},
	  $\raisebox{-12pt}{\begin{picture}(79,30)(-2,0)
			    \put(0, 0){\framebox(15,15){$1$}}
			    \put(15, 0){\framebox(15,15){}}
			    \put(30, 0){\framebox(15,15){$2$}}
			    \put(30, 15){\framebox(15,15){$3$}}
			    \put(45, 0){\framebox(15,15){}}
			    \put(60, 0){\framebox(15,15){$4$}}
			   \end{picture}}^+$,
	  $\raisebox{-12pt}{\begin{picture}(79,30)(-2,0)
			    \put(0, 0){\framebox(15,15){$1$}}
			    \put(15, 0){\framebox(15,15){}}
			    \put(30, 0){\framebox(15,15){$2$}}
			    \put(30, 15){\framebox(15,15){$3$}}
			    \put(45, 0){\framebox(15,15){}}
			    \put(60, 0){\framebox(15,15){$4$}}
			   \end{picture}}^-$,
	  $\raisebox{-12pt}{\begin{picture}(49,30)(-2,0)
			       \put(0, 0){\framebox(15,15){$1$}}
			       \put(0, 15){\framebox(15,15){$2$}}
			       \put(15, 0){\framebox(15,15){}}
			       \put(30, 0){\framebox(15,15){$3$}}
			       \put(30, 15){\framebox(15,15){$4$}}
			   \end{picture}}^+$,
	  $\raisebox{-12pt}{\begin{picture}(49,30)(-2,0)
			       \put(0, 0){\framebox(15,15){$1$}}
			       \put(0, 15){\framebox(15,15){$2$}}
			       \put(15, 0){\framebox(15,15){}}
			       \put(30, 0){\framebox(15,15){$3$}}
			       \put(30, 15){\framebox(15,15){$4$}}
			   \end{picture}}^-$,
	  \raisebox{-18pt}{\begin{picture}(49,45)(-2,0)
			     \put(0, 0){\framebox(15,15){$1$}}
			     \put(0, 15){\framebox(15,15){$2$}}
			     \put(0, 30){\framebox(15,15){$3$}}
			     \put(15, 0){\framebox(15,15){}}
			     \put(30, 0){\framebox(15,15){$4$}}
			   \end{picture}}.
	 \end{center}

  \item  two $3$-cells
	 \begin{center}
	  $\raisebox{-24pt}{\begin{picture}(19,60)(-2,0)
				 \put(0, 0){\framebox(15,15){$1$}}
				 \put(0, 15){\framebox(15,15){$2$}}
				 \put(0, 30){\framebox(15,15){$3$}}
				 \put(0, 45){\framebox(15,15){$4$}}
				\end{picture}}^+$,
	  $\raisebox{-24pt}{\begin{picture}(19,60)(-2,0)
				 \put(0, 0){\framebox(15,15){$1$}}
				 \put(0, 15){\framebox(15,15){$2$}}
				 \put(0, 30){\framebox(15,15){$3$}}
				 \put(0, 45){\framebox(15,15){$4$}}
				\end{picture}}^-$.
	 \end{center}
 \end{itemize}

\end{lemma}

We use above symbols as representatives of cells in
$\Sal(\mathcal{C}_3^2)$. We define orientations on these cells and then
transfer orientations to other cells in $\Sal(\mathcal{C}_3^2)$ via the
action of $\Sigma_4$. Those cell which are mapped to cells of the
same dimensions in $\Sal(\mathcal{A}_3)$ by the map $i_4^2$ are oriented
in such a way $i_4^2$ is orientation preserving. Then remaining four
types of cells  
\begin{center}
	  $\raisebox{-6pt}{\begin{picture}(79,15)(-2,0)
			  \put(0, 0){\framebox(15,15){$1$}}
			  \put(15, 0){\framebox(15,15){$2$}}
			  \put(30, 0){\framebox(15,15){}}
			  \put(45, 0){\framebox(15,15){$3$}}
			  \put(60, 0){\framebox(15,15){$4$}}
			 \end{picture}}^+$,
	  $\raisebox{-6pt}{\begin{picture}(79,15)(-2,0)
			  \put(0, 0){\framebox(15,15){$1$}}
			  \put(15, 0){\framebox(15,15){$2$}}
			  \put(30, 0){\framebox(15,15){}}
			  \put(45, 0){\framebox(15,15){$3$}}
			  \put(60, 0){\framebox(15,15){$4$}}
			 \end{picture}}^-$,
 $\raisebox{-12pt}{\begin{picture}(79,30)(-2,0)
			    \put(0, 0){\framebox(15,15){$1$}}
			    \put(15, 0){\framebox(15,15){}}
			    \put(30, 0){\framebox(15,15){$2$}}
			    \put(30, 15){\framebox(15,15){$3$}}
			    \put(45, 0){\framebox(15,15){}}
			    \put(60, 0){\framebox(15,15){$4$}}
			   \end{picture}}^+$,
$\raisebox{-12pt}{\begin{picture}(79,30)(-2,0)
			    \put(0, 0){\framebox(15,15){$1$}}
			    \put(15, 0){\framebox(15,15){}}
			    \put(30, 0){\framebox(15,15){$2$}}
			    \put(30, 15){\framebox(15,15){$3$}}
			    \put(45, 0){\framebox(15,15){}}
			    \put(60, 0){\framebox(15,15){$4$}}
			   \end{picture}}^-$
\end{center}
are oriented as follows. As we will see below, the first two $1$-cells
have $\raisebox{-6pt}{\begin{picture}(79,15)(-2,0) 
		       \put(0, 0){\framebox(15,15){$1$}}
		       \put(15, 0){\framebox(15,15){$2$}}
		       \put(30, 0){\framebox(15,15){$3$}}
		       \put(45, 0){\framebox(15,15){}}
		       \put(60, 0){\framebox(15,15){$4$}}
		      \end{picture}}$
 and 
 $\raisebox{-6pt}{\begin{picture}(79,15)(-2,0) 
		   \put(0, 0){\framebox(15,15){$1$}}
		   \put(15, 0){\framebox(15,15){}}
		   \put(30, 0){\framebox(15,15){$2$}}
		   \put(45, 0){\framebox(15,15){$3$}}
		   \put(60, 0){\framebox(15,15){$4$}}
		  \end{picture}}$
 as vertices. We orient these $1$-cells from
 $\raisebox{-6pt}{\begin{picture}(79,15)(-2,0) 
		   \put(0, 0){\framebox(15,15){$1$}}
		   \put(15, 0){\framebox(15,15){}}
		   \put(30, 0){\framebox(15,15){$2$}}
		   \put(45, 0){\framebox(15,15){$3$}}
		   \put(60, 0){\framebox(15,15){$4$}}
		  \end{picture}}$
to $\raisebox{-6pt}{\begin{picture}(79,15)(-2,0) 
		       \put(0, 0){\framebox(15,15){$1$}}
		       \put(15, 0){\framebox(15,15){$2$}}
		       \put(30, 0){\framebox(15,15){$3$}}
		       \put(45, 0){\framebox(15,15){}}
		       \put(60, 0){\framebox(15,15){$4$}}
		      \end{picture}}$.
The $2$-cells 
$\raisebox{-12pt}{\begin{picture}(79,30)(-2,0)
			    \put(0, 0){\framebox(15,15){$1$}}
			    \put(15, 0){\framebox(15,15){}}
			    \put(30, 0){\framebox(15,15){$2$}}
			    \put(30, 15){\framebox(15,15){$3$}}
			    \put(45, 0){\framebox(15,15){}}
			    \put(60, 0){\framebox(15,15){$4$}}
			   \end{picture}}^+$ and
$\raisebox{-12pt}{\begin{picture}(79,30)(-2,0)
			    \put(0, 0){\framebox(15,15){$1$}}
			    \put(15, 0){\framebox(15,15){}}
			    \put(30, 0){\framebox(15,15){$2$}}
			    \put(30, 15){\framebox(15,15){$3$}}
			    \put(45, 0){\framebox(15,15){}}
			    \put(60, 0){\framebox(15,15){$4$}}
			   \end{picture}}^-$
contains
$\raisebox{-12pt}{\begin{picture}(64,30)(-2,0)
			  \put(0, 0){\framebox(15,15){$1$}}
			  \put(15, 0){\framebox(15,15){$2$}}
			  \put(15, 15){\framebox(15,15){$3$}}
			  \put(30, 0){\framebox(15,15){}}
			  \put(45, 0){\framebox(15,15){$4$}}
			   \end{picture}}$
in the boundary, as we will see later. We orient these $2$-cells in such
a way the incidence number to this $1$-cell is positive.

Now we are ready to consider the boundaries. This can be done by using
the formula for the boundary in Lemma \ref{D(F,C)} and a formula
analogous to the case of the braid arrangement (Lemma
\ref{boundary_formula}).  

The first nontrivial case is the boundaries of $1$-cells.

\begin{lemma}
 \label{first_boundary}
 We have the following formula in $C_*^{4,2}$.
 \begin{eqnarray*}
  \partial_1\left(\raisebox{-12pt}{\begin{picture}(64,30)(-2,0)
	   \put(0, 0){\framebox(15,15){$1$}}
	   \put(15, 0){\framebox(15,15){}}
	   \put(30, 0){\framebox(15,15){$2$}}
	   \put(30, 15){\framebox(15,15){$3$}}
	   \put(45, 0){\framebox(15,15){$4$}}
	  \end{picture}}\right) & = & 
  2\raisebox{-6pt}{\begin{picture}(79,15)(-2,0) 
	 \put(0, 0){\framebox(15,15){$1$}}
	 \put(15, 0){\framebox(15,15){}}
	 \put(30, 0){\framebox(15,15){$2$}}
	 \put(45, 0){\framebox(15,15){$3$}}
	 \put(60, 0){\framebox(15,15){$4$}}
		   \end{picture}} \\
  \partial_1\left(\raisebox{-12pt}{\begin{picture}(64,30)(-2,0)
				\put(0, 0){\framebox(15,15){$1$}}
				\put(15, 0){\framebox(15,15){}}
				\put(30, 0){\framebox(15,15){$2$}}
				\put(45, 0){\framebox(15,15){$3$}}
				\put(45, 15){\framebox(15,15){$4$}}
			   \end{picture}}\right) & = & 
  2\raisebox{-6pt}{\begin{picture}(79,15)(-2,0) 
	 \put(0, 0){\framebox(15,15){$1$}}
	 \put(15, 0){\framebox(15,15){}}
	 \put(30, 0){\framebox(15,15){$2$}}
	 \put(45, 0){\framebox(15,15){$3$}}
	 \put(60, 0){\framebox(15,15){$4$}}
		   \end{picture}} \\
  \partial_1\left(\raisebox{-6pt}{\begin{picture}(79,15)(-2,0)
			  \put(0, 0){\framebox(15,15){$1$}}
			  \put(15, 0){\framebox(15,15){$2$}}
			  \put(30, 0){\framebox(15,15){}}
			  \put(45, 0){\framebox(15,15){$3$}}
			  \put(60, 0){\framebox(15,15){$4$}}
			 \end{picture}}^+\right) & = & 
  \raisebox{-6pt}{\begin{picture}(79,15)(-2,0) 
	 \put(0, 0){\framebox(15,15){$1$}}
	 \put(15, 0){\framebox(15,15){$2$}}
	 \put(30, 0){\framebox(15,15){$3$}}
	 \put(45, 0){\framebox(15,15){}}
	 \put(60, 0){\framebox(15,15){$4$}}
	 \end{picture}}
	 - \raisebox{-6pt}{\begin{picture}(79,15)(-2,0) 
	 \put(0, 0){\framebox(15,15){$1$}}
	 \put(15, 0){\framebox(15,15){}}
	 \put(30, 0){\framebox(15,15){$2$}}
	 \put(45, 0){\framebox(15,15){$3$}}
	 \put(60, 0){\framebox(15,15){$4$}}
	 \end{picture}} \\
  \partial_1\left(\raisebox{-6pt}{\begin{picture}(79,15)(-2,0)
			  \put(0, 0){\framebox(15,15){$1$}}
			  \put(15, 0){\framebox(15,15){$2$}}
			  \put(30, 0){\framebox(15,15){}}
			  \put(45, 0){\framebox(15,15){$3$}}
			  \put(60, 0){\framebox(15,15){$4$}}
			 \end{picture}}^-\right) & = & 
\raisebox{-6pt}{\begin{picture}(79,15)(-2,0) 
	 \put(0, 0){\framebox(15,15){$1$}}
	 \put(15, 0){\framebox(15,15){$2$}}
	 \put(30, 0){\framebox(15,15){$3$}}
	 \put(45, 0){\framebox(15,15){}}
	 \put(60, 0){\framebox(15,15){$4$}}
	 \end{picture}}
	 - \raisebox{-6pt}{\begin{picture}(79,15)(-2,0) 
	 \put(0, 0){\framebox(15,15){$1$}}
	 \put(15, 0){\framebox(15,15){}}
	 \put(30, 0){\framebox(15,15){$2$}}
	 \put(45, 0){\framebox(15,15){$3$}}
	 \put(60, 0){\framebox(15,15){$4$}}
	 \end{picture}} \\
  \partial_1\left(\raisebox{-12pt}{\begin{picture}(64,30)(-2,0)
			  \put(0, 0){\framebox(15,15){$1$}}
			  \put(15, 0){\framebox(15,15){$2$}}
			  \put(15, 15){\framebox(15,15){$3$}}
			  \put(30, 0){\framebox(15,15){}}
			  \put(45, 0){\framebox(15,15){$4$}}
			   \end{picture}}\right) & = &
 2\raisebox{-6pt}{\begin{picture}(79,15)(-2,0) 
	 \put(0, 0){\framebox(15,15){$1$}}
	 \put(15, 0){\framebox(15,15){$2$}}
	 \put(30, 0){\framebox(15,15){$3$}}
	 \put(45, 0){\framebox(15,15){}}
	 \put(60, 0){\framebox(15,15){$4$}}
	 \end{picture}} \\
  \partial_1\left(\raisebox{-12pt}{\begin{picture}(64,30)(-2,0)
			  \put(0, 0){\framebox(15,15){$1$}}
			  \put(0, 15){\framebox(15,15){$2$}}
			  \put(15, 0){\framebox(15,15){$3$}}
			  \put(30, 0){\framebox(15,15){}}
			  \put(45, 0){\framebox(15,15){$4$}}
			   \end{picture}}\right) & = & 
2\raisebox{-6pt}{\begin{picture}(79,15)(-2,0) 
	 \put(0, 0){\framebox(15,15){$1$}}
	 \put(15, 0){\framebox(15,15){$2$}}
	 \put(30, 0){\framebox(15,15){$3$}}
	 \put(45, 0){\framebox(15,15){}}
	 \put(60, 0){\framebox(15,15){$4$}}
	 \end{picture}}.
 \end{eqnarray*}
\end{lemma}

\begin{proof}
 By Lemma \ref{D(F,C)}
 \[
  \partial_1\left(\raisebox{-12pt}{\begin{picture}(64,30)(-2,0)
	   \put(0, 0){\framebox(15,15){$1$}}
	   \put(15, 0){\framebox(15,15){}}
	   \put(30, 0){\framebox(15,15){$2$}}
	   \put(30, 15){\framebox(15,15){$3$}}
	   \put(45, 0){\framebox(15,15){$4$}}
	  \end{picture}}\right) = \sum_{F}
 \left[\raisebox{-12pt}{\begin{picture}(64,30)(-2,0)
	   \put(0, 0){\framebox(15,15){$1$}}
	   \put(15, 0){\framebox(15,15){}}
	   \put(30, 0){\framebox(15,15){$2$}}
	   \put(30, 15){\framebox(15,15){$3$}}
	   \put(45, 0){\framebox(15,15){$4$}}
	  \end{picture}}:
 D\left(F,F\circ
 \raisebox{-6pt}{\begin{picture}(79,15)(-2,0)
	 \put(0, 0){\framebox(15,15){$1$}}
	 \put(15, 0){\framebox(15,15){}}
	 \put(30, 0){\framebox(15,15){$2$}}
	 \put(45, 0){\framebox(15,15){$3$}}
	 \put(60, 0){\framebox(15,15){$4$}}
		   \end{picture}}\right)
 \right]
 D\left(F,F\circ
 \raisebox{-6pt}{\begin{picture}(79,15)(-2,0)
	 \put(0, 0){\framebox(15,15){$1$}}
	 \put(15, 0){\framebox(15,15){}}
	 \put(30, 0){\framebox(15,15){$2$}}
	 \put(45, 0){\framebox(15,15){$3$}}
	 \put(60, 0){\framebox(15,15){$4$}}
		   \end{picture}}\right)
 \]
where $F$ runs over all faces containing 
\raisebox{-12pt}{\begin{picture}(64,30)(-2,0)
	   \put(0, 0){\framebox(15,15){$1$}}
	   \put(15, 0){\framebox(15,15){}}
	   \put(30, 0){\framebox(15,15){$2$}}
	   \put(30, 15){\framebox(15,15){$3$}}
	   \put(45, 0){\framebox(15,15){$4$}}
	  \end{picture}} in $\mathcal{L}(\mathcal{C}^2_3)$. 
 In this case, we have
 \begin{eqnarray*}
   \partial_1\left(\raisebox{-12pt}{\begin{picture}(64,30)(-2,0)
	   \put(0, 0){\framebox(15,15){$1$}}
	   \put(15, 0){\framebox(15,15){}}
	   \put(30, 0){\framebox(15,15){$2$}}
	   \put(30, 15){\framebox(15,15){$3$}}
	   \put(45, 0){\framebox(15,15){$4$}}
	  \end{picture}}\right) & = &
 \varepsilon_1 D\left(\raisebox{-6pt}{\begin{picture}(79,15)(-2,0)
	 \put(0, 0){\framebox(15,15){$1$}}
	 \put(15, 0){\framebox(15,15){}}
	 \put(30, 0){\framebox(15,15){$2$}}
	 \put(45, 0){\framebox(15,15){$3$}}
	 \put(60, 0){\framebox(15,15){$4$}}
		   \end{picture}},
 \raisebox{-6pt}{\begin{picture}(79,15)(-2,0)
	 \put(0, 0){\framebox(15,15){$1$}}
	 \put(15, 0){\framebox(15,15){}}
	 \put(30, 0){\framebox(15,15){$2$}}
	 \put(45, 0){\framebox(15,15){$3$}}
	 \put(60, 0){\framebox(15,15){$4$}}
		   \end{picture}}\circ
 \raisebox{-6pt}{\begin{picture}(79,15)(-2,0)
	 \put(0, 0){\framebox(15,15){$1$}}
	 \put(15, 0){\framebox(15,15){}}
	 \put(30, 0){\framebox(15,15){$2$}}
	 \put(45, 0){\framebox(15,15){$3$}}
	 \put(60, 0){\framebox(15,15){$4$}}
		   \end{picture}}\right) \\
  & + & \varepsilon_2 D\left(\raisebox{-6pt}{\begin{picture}(79,15)(-2,0)
	 \put(0, 0){\framebox(15,15){$1$}}
	 \put(15, 0){\framebox(15,15){}}
	 \put(30, 0){\framebox(15,15){$3$}}
	 \put(45, 0){\framebox(15,15){$2$}}
	 \put(60, 0){\framebox(15,15){$4$}}
		   \end{picture}},
 \raisebox{-6pt}{\begin{picture}(79,15)(-2,0)
	 \put(0, 0){\framebox(15,15){$1$}}
	 \put(15, 0){\framebox(15,15){}}
	 \put(30, 0){\framebox(15,15){$3$}}
	 \put(45, 0){\framebox(15,15){$2$}}
	 \put(60, 0){\framebox(15,15){$4$}}
		   \end{picture}}\circ
 \raisebox{-6pt}{\begin{picture}(79,15)(-2,0)
	 \put(0, 0){\framebox(15,15){$1$}}
	 \put(15, 0){\framebox(15,15){}}
	 \put(30, 0){\framebox(15,15){$2$}}
	 \put(45, 0){\framebox(15,15){$3$}}
	 \put(60, 0){\framebox(15,15){$4$}}
		   \end{picture}}\right),
 \end{eqnarray*}
 where $\varepsilon_1,\varepsilon_2$ are appropriate incidence numbers.

 In order to compute the boundaries, therefore, we need to understand
 the matroid product in the face lattice $\mathcal{L}(\mathcal{C}_3^2)$.

 As we have recalled in \S\ref{Salvetti_complex}, the cellular structure
 of the Salvetti complex was originally described by face-chamber
 pairings. For computations, however, it is much more convenient to
 regard faces as functions from the set of normal vectors to the poset
 of three elements $S_1=\{0,+1,-1\}$ and use the matroid product of these
 functions, as we have 
 seen in Lemma \ref{FaceLattice} and \ref{MatroidProduct}.

 We choose the following set of normal vectors for the arrangement
 $\mathcal{C}_3^2$:
 \begin{eqnarray*}
  \bm{a}_1 & = & (1,-1,0,0) \\
  \bm{a}_2 & = & (0,1,-1,0) \\
  \bm{a}_3 & = & (0,0,1,-1) \\
  \bm{a}_4 & = & (1,0,-1,0) \\
  \bm{a}_5 & = & (1,0,0,-1) \\
  \bm{a}_6 & = & (0,1,0,-1) \\
  \bm{a}_7 & = & (1,1,-1,-1) \\
  \bm{a}_8 & = & (1,-1,1,-1) \\
  \bm{a}_9 & = & (1,-1,-1,1).
 \end{eqnarray*}
 Then a face $F \in \mathcal{L}(\mathcal{C}_3^2)$ can be regarded as a
 function
 \[
  \tau_F : \{\bm{a}_1,\bm{a}_2,\cdots,\bm{a}_9\} \longrightarrow
 S_1.
 \]
 For simplicity, we denote elements $+1,-1$ in $S_1$ by $+, -$,
 respectively, and denote the above function by the symbol
 \[
  (\tau_F(\bm{a}_1),\tau_F(\bm{a}_2),\cdots,\tau_F(\bm{a}_6) \mid
 \tau_F(\bm{a}_7),\tau_F(\bm{a}_8), \tau_F(\bm{a}_9)).
 \]
 For example, the face 
\raisebox{-6pt}{
} corresponds to the symbol
 \[
  (-,-,-,-,-,-\mid -,-,-).
 \]
 With these notations, the matroid product of a face and a chamber is
 given by replacing $0$'s in the face by the values at the same position
 in the chamber. For example,
 \begin{eqnarray*}
\raisebox{-6pt}{%
}.
 \end{eqnarray*}
 Thus, by taking the orientations into an account, we have
 \begin{eqnarray*}
  \partial_1\left(\raisebox{-12pt}{%
}.
 \end{eqnarray*}
 By analogous calculations, we obtain
 \begin{eqnarray*}
  \partial_1\left(\raisebox{-12pt}{%
}.
 \end{eqnarray*}

 On the other hand,
 $\raisebox{-6pt}{%
})$ as its boundary. By computing the matroid
 products, we see these $0$-cells are 
$\raisebox{-6pt}{%
}$, respectively.
By the definition of the orientation, we obtain

 \begin{eqnarray*}
    \partial_1\left(\raisebox{-6pt}{%
}.
 \end{eqnarray*}

\end{proof}

\begin{remark}
 \label{same_as_braid}
 Note that in the above proof, the computation of
 $\partial_1\left(\raisebox{-12pt}{%
}\right)$
 is exactly the same as that of
 $\partial_1\left(\raisebox{-12pt}{%
}\right)$
 in $C_*(\Sal(\mathcal{A}_3))\otimes_{\Sigma_4}\F_p(\pm 1)$.

 In general, boundaries of those cells which do not hav $+$ or $-$ sign
 on the shoulder can be computed by the same formulas for the
 corresponding cells in the braid arrangement.
\end{remark}

Let us consider $2$-cells next.

\begin{lemma}
 We have the following formula in $C_*^{4,2}$.
 \begin{eqnarray*}
  \partial_2\left(\raisebox{-18pt}{%
}.
 \end{eqnarray*}
\end{lemma}

\begin{proof}
 The computations of
 $\partial_2\left(\raisebox{-18pt}{%
}\right)$
are essentially the same as those of
 $\partial_2\left(\raisebox{-20pt}{%
}\right)$ in \S\ref{FC4} as is noted in Remark
 \ref{same_as_braid} and are omitted.

 There are four faces in $\mathcal{L}(\mathcal{C}_3^2)$ that contain
 \raisebox{-12pt}{%
} as faces
 Thus
 \begin{eqnarray*}
   \partial_2\left(\raisebox{-12pt}{%
}^+,
 \end{eqnarray*}
 where $\varepsilon_i \in \{\pm 1\}$ ($i=1,2,3,4$) are certain signs. By
 the definition of the orientation of 
 $\raisebox{-12pt}{%
}^+$, we have
 $\varepsilon_2=1$. Since $\partial_1\circ\partial_2=0$, we see that
 $\varepsilon_1=-1$ and $\varepsilon_3+\varepsilon_4=2$.

 By the same calculation, we also have
 \[
    \partial_2\left(\raisebox{-12pt}{%
}^{-}.
 \]

 Finally, we have
 \begin{eqnarray*}
    \partial_2\left(\raisebox{-12pt}{%
}^-,
 \end{eqnarray*}
 where coefficients (signs) of the first two terms are determined by comparing
 with the braid arrangement. The coefficients of the last two terms are
 determined in the same way as in the previous calculation.

 The calculation of
 $\partial_2\left(\raisebox{-12pt}{%
}^-\right)$
can be done in the same manner.
\end{proof}

It remains to compute boundaries on $3$-cells.

\begin{lemma}
 \label{boundary_of_3-cells}
 We have
 \begin{eqnarray*}
  \partial_3\left(\raisebox{-28pt}{%
}^-
 \end{eqnarray*}
\end{lemma}

\begin{proof}
 Let us consider 
 $\partial_3\left(\raisebox{-28pt}{%
}^+\right)$.
There are many faces containing 
 \raisebox{-28pt}{%
}.
 For example, there are $\binom{4}{2}=6$ faces of the shape
\raisebox{-12pt}{%
}.
Matroid products can be computed, for example, 
 \begin{eqnarray*}
  \raisebox{-12pt}{%
}.
 \end{eqnarray*}
 Similarly, we have
 \begin{eqnarray*}
  \raisebox{-12pt}{%
}
 \end{eqnarray*}

 This implies that, in
 $C_3(\Sal(\mathcal{C}_3^2))\otimes_{\Sigma_4}\F_p$, if we write
 \begin{eqnarray*}
   \partial_3\left(\raisebox{-28pt}{%
},
 \end{eqnarray*}
 we have
 \[
  A_4 = -5,\ A_5 = -1.
 \]
 By similar calculations of matroid products and by that fact
 $\partial_2\partial_3=0$, we can determine other 
 coefficients and we have
 \begin{eqnarray*}
  \partial_3\left(\raisebox{-28pt}{%
}^-.
 \end{eqnarray*}
 Analogously we have
 \begin{eqnarray*}
  \partial_3\left(\raisebox{-28pt}{%
}^-.
 \end{eqnarray*}
\end{proof}

\subsection{The homology of $\mathcal{C}_3^2$}
\label{M2C4}

In this section, we compare
$H_*(\Sal(\mathcal{C}_3^2)\otimes_{\Sigma_4}\F_p(\pm 1))$ and 
$H_*(\Sal(\mathcal{A}_3)\otimes_{\Sigma_4}\F_p(\pm 1))$ for $p$ a
prime, following the strategy described in the beginning of
\S\ref{center}. 

We have analyzed the cell structure of $\Sal(\mathcal{C}_3^2)$ in the
previous section based on the structure of
$\mathcal{L}(\mathcal{C}_3^2)/\Sigma_4$ investigated in
\S\ref{cells}. 

In this section, we compute the homology of
\[
  K_*^{4,2} = \Ker(i_4^2 : C_*(\Sal(\mathcal{C}_3^2))\otimes_{\Sigma_4}
  \F_p(\pm 1) \to C_*(\Sal(\mathcal{A}_3))\otimes_{\Sigma_4}
  \F_p(\pm 1))
\]
for an odd prime $p$. We first need to know generators for $K_*^{4,2}$.

The generators of the cellular chain complex 
$C_*(\Sal(\mathcal{C}_3^2))\otimes_{\Sigma_4}\F_p(\pm 1)$ are in
one-to-one correspondence with cells in
$\Sal(\mathcal{C}_3^2)/\Sigma_4$ even when $p$ is odd. Of course, we
have to take the sign representation into account, when we compute the
boundary homomorphisms.

\begin{lemma}
 \label{generators_of_K42}
 Define
 \begin{eqnarray*}
  x_0 & = &
    \raisebox{-6pt}{\begin{picture}(79,15)(-2,0)
		     \put(0, 0){\framebox(15,15){$1$}}
		     \put(15, 0){\framebox(15,15){$2$}}
		     \put(30, 0){\framebox(15,15){$3$}}
		     \put(45, 0){\framebox(15,15){}}
		     \put(60, 0){\framebox(15,15){$4$}}
		    \end{picture}}- 
		    \raisebox{-6pt}{\begin{picture}(79,15)(-2,0)
				     \put(0, 0){\framebox(15,15){$1$}}
				     \put(15, 0){\framebox(15,15){}}
				     \put(30, 0){\framebox(15,15){$2$}}
				     \put(45, 0){\framebox(15,15){$3$}}
				     \put(60, 0){\framebox(15,15){$4$}}
				    \end{picture}} \\
  x_{11} & = &  \raisebox{-12pt}{\begin{picture}(64,30)(-2,0)
			  \put(0, 0){\framebox(15,15){$1$}}
			  \put(15, 0){\framebox(15,15){$2$}}
			  \put(15, 15){\framebox(15,15){$3$}}
			  \put(30, 0){\framebox(15,15){}}
			  \put(45, 0){\framebox(15,15){$4$}}
			   \end{picture}}
-\raisebox{-12pt}{\begin{picture}(64,30)(-2,0)
	   \put(0, 0){\framebox(15,15){$1$}}
	   \put(15, 0){\framebox(15,15){}}
	   \put(30, 0){\framebox(15,15){$2$}}
	   \put(30, 15){\framebox(15,15){$3$}}
	   \put(45, 0){\framebox(15,15){$4$}}
			  \end{picture}}  \\
  x_{12}^+ & = & \raisebox{-6pt}{\begin{picture}(79,15)(-2,0)
			  \put(0, 0){\framebox(15,15){$1$}}
			  \put(15, 0){\framebox(15,15){$2$}}
			  \put(30, 0){\framebox(15,15){}}
			  \put(45, 0){\framebox(15,15){$3$}}
			  \put(60, 0){\framebox(15,15){$4$}}
			   \end{picture}}^+ \\
  x_{12}^- & = & \raisebox{-6pt}{\begin{picture}(79,15)(-2,0)
			  \put(0, 0){\framebox(15,15){$1$}}
			  \put(15, 0){\framebox(15,15){$2$}}
			  \put(30, 0){\framebox(15,15){}}
			  \put(45, 0){\framebox(15,15){$3$}}
			  \put(60, 0){\framebox(15,15){$4$}}
			 \end{picture}}^- \\
  x_{21} & = & \raisebox{-12pt}{\begin{picture}(49,30)(-2,0)
			       \put(0, 0){\framebox(15,15){$1$}}
			       \put(0, 15){\framebox(15,15){$2$}}
			       \put(15, 0){\framebox(15,15){}}
			       \put(30, 0){\framebox(15,15){$3$}}
			       \put(30, 15){\framebox(15,15){$4$}}
			   \end{picture}}^+ -
	  \raisebox{-12pt}{\begin{picture}(49,30)(-2,0)
			       \put(0, 0){\framebox(15,15){$1$}}
			       \put(0, 15){\framebox(15,15){$2$}}
			       \put(15, 0){\framebox(15,15){}}
			       \put(30, 0){\framebox(15,15){$3$}}
			       \put(30, 15){\framebox(15,15){$4$}}
			   \end{picture}}^- \\
  x_{22}^+ & = & \raisebox{-12pt}{\begin{picture}(79,30)(-2,0)
			       \put(0, 0){\framebox(15,15){$1$}}
			       \put(15, 0){\framebox(15,15){}}
			       \put(30, 0){\framebox(15,15){$2$}}
			       \put(30, 15){\framebox(15,15){$3$}}
			       \put(45, 0){\framebox(15,15){}}
			       \put(60, 0){\framebox(15,15){$4$}}
			    \end{picture}}^+ \\ 
  x_{22}^- & = & \raisebox{-12pt}{\begin{picture}(79,30)(-2,0)
			    \put(0, 0){\framebox(15,15){$1$}}
			    \put(15, 0){\framebox(15,15){}}
			    \put(30, 0){\framebox(15,15){$2$}}
			    \put(30, 15){\framebox(15,15){$3$}}
			    \put(45, 0){\framebox(15,15){}}
			    \put(60, 0){\framebox(15,15){$4$}}
		   \end{picture}}^- \\
  x_3  & = & \raisebox{-24pt}{\begin{picture}(19,60)(-2,0)
				 \put(0, 0){\framebox(15,15){$1$}}
				 \put(0, 15){\framebox(15,15){$2$}}
				 \put(0, 30){\framebox(15,15){$3$}}
				 \put(0, 45){\framebox(15,15){$4$}}
			 \end{picture}}^+ - 
	  \raisebox{-24pt}{\begin{picture}(19,60)(-2,0)
				 \put(0, 0){\framebox(15,15){$1$}}
				 \put(0, 15){\framebox(15,15){$2$}}
				 \put(0, 30){\framebox(15,15){$3$}}
				 \put(0, 45){\framebox(15,15){$4$}}
				\end{picture}}^-.
 \end{eqnarray*}
 Then these are generators for $K_*^{4,2}$
 \begin{eqnarray*}
  K_0^{4,2} & = & \left\langle x_0 \right\rangle, \\
  K_1^{4,2} & = & \left\langle x_{11}, x_{12}^+, x_{12}^{-}
		  \right\rangle, \\ 
  K_2^{4,2} & = & \left\langle x_{21}, x_{22}^{+}, x_{22}^{-}
		  \right\rangle, \\ 
  K_3^{4,2} & = & \left\langle x_3 \right\rangle.
 \end{eqnarray*}
\end{lemma}

Thanks to the calculations in the previous section, we can easily
compute the boundaries on these generators.

\begin{lemma}
 \label{boundaries_on_K42}
  Boundaries are given by
 \begin{eqnarray*}
  \partial_1(x_{11}) & = & 2x_0 \\
  \partial_1(x_{12}^+) & = & x_0 \\
  \partial_1(x_{12}^-) & = & x_0 \\
  \partial_2(x_{21}) & = & 2(x_{12}^+-x_{12}^-) \\
  \partial_2(x_{22}^+) & = & x_{11}-2x_{12}^{+} \\
  \partial_2(x_{22}^-) & = & x_{11}-2x_{12}^{-} \\
  \partial_3(x_3) & = & -4x_{21}+4(x_{22}^+-x_{22}^-).
 \end{eqnarray*}
\end{lemma}

By an elementary calculation, we obtain the homology of $K_*^{4,2}$.

\begin{proposition}
 When $p$ is odd,
 \[
  H_i(K_*^{4,2}) \cong 0
 \]
 for all $i$.
\end{proposition}

As a corollary, we obtain Theorem 
\ref{p>=3}.

\begin{corollary}[Theorem \ref{p>=3}]
 For an odd prime $p$, the inclusion $M_2(\C,4) \hookrightarrow F(\C,4)$
 induces an isomorphism
 \[
  H_*(S_*(M_2(\C,4))\otimes_{\Sigma_4}\F_p(\pm 1)) \cong
 H_*(S_*(F(\C,4))\otimes_{\Sigma_4}\F_p(\pm 1)). 
 \]
\end{corollary}

When $p=2$, Lemma \ref{boundaries_on_K42} implies that the boundaries in
$K_*^{4,2}$ are given by.
 \begin{eqnarray*}
  \partial_1(x_{11}) & = & 0 \\
  \partial_1(x_{12}^+) & = & x_0 \\
  \partial_1(x_{12}^-) & = & x_0 \\
  \partial_2(x_{21}) & = & 0 \\
  \partial_2(x_{22}^+) & = & x_{11} \\
  \partial_2(x_{22}^-) & = & x_{11} \\
  \partial_3(x_3) & = & 0.
 \end{eqnarray*}
In particular, $x_3$ represents a nontrivial cycle in $C_3^{4,2}$ that
is mapped to $0$ under the map $i_4^2$. Thus we obtain a proof of
Theorem \ref{fails_at_2}.

\bibliographystyle{halpha}
\bibliography{%
\bibdir/mathAb,%
\bibdir/mathAl,%
\bibdir/mathAn,%
\bibdir/mathA,%
\bibdir/mathB,%
\bibdir/mathBa,%
\bibdir/mathBe,%
\bibdir/mathBo,%
\bibdir/mathBr,%
\bibdir/mathBu,%
\bibdir/mathCa,%
\bibdir/mathCh,%
\bibdir/mathCo,%
\bibdir/mathC,%
\bibdir/mathDa,%
\bibdir/mathDe,%
\bibdir/mathDu,%
\bibdir/mathD,%
\bibdir/mathE,%
\bibdir/mathFa,%
\bibdir/mathFo,%
\bibdir/mathFr,%
\bibdir/mathF,%
\bibdir/mathGa,%
\bibdir/mathGo,%
\bibdir/mathGr,%
\bibdir/mathG,%
\bibdir/mathHa,%
\bibdir/mathHe,%
\bibdir/mathH,%
\bibdir/mathI,%
\bibdir/mathJ,%
\bibdir/mathKa,%
\bibdir/mathKo,%
\bibdir/mathK,%
\bibdir/mathLa,%
\bibdir/mathLe,%
\bibdir/mathLu,%
\bibdir/mathL,%
\bibdir/mathMa,%
\bibdir/mathMc,%
\bibdir/mathMi,%
\bibdir/mathM,%
\bibdir/mathN,%
\bibdir/mathO,%
\bibdir/mathP,%
\bibdir/mathQ,%
\bibdir/mathR,%
\bibdir/mathSa,%
\bibdir/mathSe,%
\bibdir/mathSt,%
\bibdir/mathS,%
\bibdir/mathT,%
\bibdir/mathU,%
\bibdir/mathV,%
\bibdir/mathW,%
\bibdir/mathX,%
\bibdir/mathY,%
\bibdir/mathZ,%
\bibdir/physics,%
\bibdir/personal}

\end{document}